\documentclass[a4paper,11pt,leqno]{amsart}
\calclayout
\usepackage[matrix,arrow,tips,curve]{xy}
\usepackage{pictexwd, dcpic} 
\usepackage[english]{babel}
\usepackage{amsmath}
\usepackage{amssymb}    
\usepackage{graphicx}
\usepackage{mathrsfs}
\usepackage{enumerate} 
\usepackage{psfrag}

\usepackage{natbib}

 \psfrag{E}{\fontsize{20}{20}$E$}
 \psfrag{F}{\fontsize{20}{20}$F$}
 \psfrag{B}{\fontsize{20}{20}$\Delta_0$}
 \psfrag{l}{\fontsize{20}{20}$\ell$}
 \psfrag{E5}{\fontsize{20}{20}$\wi{E}$}
 \psfrag{E0}{\fontsize{20}{20}$E_0$}
 \psfrag{E1}{\fontsize{20}{20}$E_1$}
 \psfrag{E2}{\fontsize{20}{20}$E_2$}
 \psfrag{A1}{\fontsize{20}{20}$A_1$}
 \psfrag{A2}{\fontsize{20}{20}$A_2$}
 \psfrag{G1}{\fontsize{20}{20}$E_0'$}
 \psfrag{G}{\fontsize{20}{20}$E_0'$}
 \psfrag{G2}{\fontsize{20}{20}$E_0''$}
 \psfrag{D1}{\fontsize{20}{20}$D_1$}
 \psfrag{D}{\fontsize{20}{20}$D$}
 \psfrag{D2}{\fontsize{20}{20}$D_2$}
 \psfrag{R1}{\fontsize{20}{20}$R_1$}
 \psfrag{R2}{\fontsize{20}{20}$\w{R}_2$}
 \psfrag{R3}{\fontsize{20}{20}$\w{R}_3$}
 \psfrag{X3}{\fontsize{20}{20}$X_3$}
 \psfrag{R4}{\fontsize{20}{20}$\w{R}_4$}
 \psfrag{E3}{\fontsize{20}{20}$E_3$}
 \psfrag{E4}{\fontsize{20}{20}$E_4$}
 \psfrag{S3}{\fontsize{20}{20}$\sigma(E_3)$}
 \psfrag{S4}{\fontsize{20}{20}$\sigma(E_4)$}

\newtheorem{thm}{Theorem}[section]
\newtheorem{lemma}[thm]{Lemma}
\newtheorem{proposition}[thm]{Proposition}
\newtheorem{corollary}[thm]{Corollary}

\theoremstyle{definition}
\newtheorem{remark}[thm]{Remark}

\newcommand{\w}{\widetilde}

\newcommand{\wi}{\widehat}

\newcommand{\Sing}{\operatorname{Sing}}
\newcommand{\NE}{\operatorname{NE}}
\newcommand{\Exc}{\operatorname{Exc}}
\newcommand{\Lo}{\operatorname{Locus}}
\newcommand{\codim}{\operatorname{codim}}

\title[Non-elementary Fano conic bundles]{Non-elementary Fano conic bundles}
\author{E. A.~Romano}
\address{Universit\`a di Torino,
Dipartimento di Matematica,
via Carlo Alberto 10,
10123 Torino - Italy}
\email{eleonoraanna.romano@unito.it}

\begin{document}
\maketitle

\begin{abstract}
Let $X$ be a complex, projective, smooth and Fano variety. We study \textit{Fano conic bundles} $f\colon X\to Y$. Denoting by $\rho_{X}$ the Picard number of $X$, we investigate such contractions when $\rho_{X}-\rho_{Y}>1$, called \textit{non-elementary}. We prove that $\rho_{X}-\rho_{Y}\leq 8$, and we deduce new geometric information about our varieties $X$ and $Y$, depending on $\rho_{X}-\rho_{Y}$. Using our results, we show that some known examples of Fano conic bundles are elementary.
Moreover, when we allow that $X$ is locally factorial with canonical singularities and with at most finitely many non-terminal points, and $f\colon X\to Y$ is a fiber type $K_{X}$-negative contraction with one-dimensional fibers, we show that $\rho_{X}-\rho_{Y}\leq 9$.
\end{abstract}

\section{Introduction}

\noindent Let $X$ be a complex, projective, smooth and Fano variety of dimension $n$. In this paper we study a particular kind of fiber type contraction $f\colon X\to Y$, called \textit{Fano conic bundle}. This means that $f$ is a morphism between smooth projective varieties where every fiber is isomorphic to a conic in the projective plane. The main references for the background material on conic bundles are given in Section 2. \\

Let us denote by $\mathcal{N}_{1}(X)$ the $\mathbb{R}$-vector space of one-cycles with real coefficients, modulo numerical equivalence, whose dimension is the \textit{Picard number} $\rho_{X}$. Let us consider the convex cone $\NE{(X)}$ in $\mathcal{N}_{1}(X)$, spanned by the classes of effective curves. Being $X$ a Fano manifold, by the Cone Theorem it follows that $\NE{(X)}$ is closed, polyhedral, and it is spanned by finitely many classes of rational curves.

Let us denote by $\NE{(f)}$ the \textit{relative cone} of $f$, that is the convex subcone of $\NE{(X)}$ containing all classes of curves that are contracted by the conic bundle. In our case, one has that $\text{dim}{\ \NE(f)}=\rho_{X}-\rho_{Y}$.\\

The aim of this paper is the study of Fano conic bundles where the dimension of the relative cone is greater than one, that are called \textit{non-elementary}.
 
Our purpose is to deduce geometric information on our varieties and properties of the conic bundle, depending on the dimension of the relative cone. In particular, the study of non-elementary Fano conic bundle in higher dimension could be useful to provide a description of Fano manifolds which admit a conic bundle structure. \\

In \cite[$\S$4]{WIS} Wi$\acute{\text{s}}$niewski studied the following problem: given a Fano conic bundle $f\colon X\to Y$, is $Y$ Fano or not? The following theorem is the main result of this paper (see Theorem \ref{TEOPRINC} for a slightly stronger statement), and allows us to give answers, in some cases. 
\begin{thm} \label{ThmINTRODUZ}
Let $f\colon X\to Y$ be a Fano conic bundle. Then $\rho_{X}-\rho_{Y}\leq 8$. Moreover:
\begin{enumerate} [(1)]
\item if $\rho_{X}-\rho_{Y}\geq 4$, then $X\cong S\times T$, where $S$ is a del Pezzo surface, $T$ is an $(n-2)$-dimensional Fano manifold, $Y\cong \mathbb{P}^{1}\times T$ so that $Y$ is Fano, and $f$ is induced by a conic bundle on $S$.
\item If $\rho_{X}-\rho_{Y}=3$, then $f$ has only reduced fibers, $Y$ is a Fano variety, and there exists a smooth $\mathbb{P}^{1}$-fibration\footnote{A smooth $\mathbb{P}^{1}$-fibration is a smooth morphism with fibers isomorphic to $\mathbb{P}^{1}$.} $\xi\colon Y\to Y^{\prime}$, where $Y^{\prime}$ is smooth and Fano. 
\item If $\rho_{X}-\rho_{Y}=2$ and $Y$ is not Fano, then $\rho_{Y}\geq 3$ and there exists a smooth $\mathbb{P}^{1}$-fibration $\xi\colon Y\to Y^{\prime}$, where $Y^{\prime}$ is smooth.
\end{enumerate}
\end{thm}

An immediate consequence of Theorem \ref{ThmINTRODUZ} is the following result which shows that the target of a non-elementary Fano conic bundle is often Fano.

\begin{corollary} \label{esempio} Let $f\colon X\to Y$ be a Fano conic bundle, with $Y$ not Fano. Then $\rho_{X}-\rho_{Y}=2$ or $\rho_{X}-\rho_{Y}=1$. If $\rho_{X}-\rho_{Y}=2$, then $\rho_{Y}\geq 3$ and $Y$ has a smooth $\mathbb{P}^{1}$-fibration.
\end{corollary}

Note that the bound $\rho_{X}-\rho_{Y}\leq 8$ of Theorem \ref{ThmINTRODUZ} is sharp, because we get the equality when $f\colon X\to \mathbb{P}^{1}$ is a conic bundle on a del Pezzo surface $X$ with $\rho_{X}=9$. Moreover, there exists a Fano 3-fold $X$ with a conic bundle $f\colon X\to Y$ where $X$ is not a product and $\rho_{X}-\rho_{Y}=3$ (see $\S$\ref{noproduct}), so that also the bound $\rho_{X}-\rho_{Y}\geq 4$ is sharp. 

We do not know whether case $(3)$ of Theorem \ref{ThmINTRODUZ} can happen, on the other hand, there exist Fano conic bundles $f\colon X\to Y$, where $Y$ is not Fano and $\rho_{X}-\rho_{Y}=1$. In $\S$\ref{elementaryex} we first review Wi$\acute{\text{s}}$niewski's example of a Fano conic bundle $f\colon X\to Y$ in which $Y$ is not a Fano variety (see \cite[Example $\S 4$]{WIS}), then using Corollary \ref{esempio} we deduce that $f$ is elementary.\\

Besides the smooth case, we will also analyse the case in which $X$ has some mild singularities, where we prove the following:
\begin{thm} \label{singular} Let $X$ be a locally factorial, projective, Fano variety with canonical singularities and with at most finitely many non-terminal points. Let $f\colon X\to Y$ be a fiber type contraction with one-dimensional fibers. Then $\rho_{X}-\rho_{Y}\leq 9$.
\end{thm}
By the above theorem it follows that when we allow some mild singularities for $X$, there could exist examples of fiber type contractions with one-dimensional fibers and relative Picard dimension $9$. The author does not know if this can happen.\\ 

Let us describe in more detail the content of this paper. In Section 2 we set up notation and terminology, and we present some preliminaries on conic bundles and their most important geometric properties. 

We conclude this section by recalling some basic facts and results related to the Minimal Model Program (MMP) for divisors in Fano manifold, and we summarize without proofs the relevant material that will be needed.\\

In Section 3, we discuss some results that will be essential to investigate conic bundles in the next sections. In particular, in Propositions \ref{TEO1} and \ref{smoothcase} we show that there exists a particular factorization of our Fano conic bundle $f\colon X\to Y$ into elementary contractions. More precisely, if $r:=\rho_{X}-\rho_{Y}$, using this factorization we find $r-1$ smooth prime divisors $A_{1},\dots, A_{r-1}$ of $Y$, pairwise disjoint. We prove that the fibers of $f$ over $A_{i}$ are reducible and that for every $i=1,\dots,r-1$, $f^{*}(A_{i})=E_{i}+\hat{E}_{i}$, where $E_{i}$, $\hat{E}_{i}$ are prime divisors of $X$, which play an essential role throughout the paper. 

Another main geometric ingredient is represented by an invariant of $X$, the \textit{Lefschetz defect} $\delta_{X}$, whose definition is the following: 
\begin{center}
$\delta_{X}:= \text{max} \{\dim \text{coker}(i_{*}\colon \mathcal{N}_{1}(D)\to \mathcal{N}_{1}(X))\mid D \hookrightarrow X \text{ a prime divisor}\}$.
\end{center} 
This invariant was introduced in \cite{CAS1} in the smooth case, where the author proved that $\delta_{X}\leq 8$. We refer the reader to \cite{CAS1} and \cite{NOCE} for the properties of $\delta_{X}$. 
\begin{thm} [\cite{CAS1}, Theorem 1.1; \cite{NOCE}, Theorem 1.2] \label{prodotto}  Let $X$ be a $\mathbb{Q}$-factorial Gorenstein Fano variety with canonical singularities, and with at most finitely many non-terminal points. Then $\delta_{X}\leq 8$. Moreover, if $X$ is smooth and $\delta_{X}\geq 4$ then $X\cong S\times T$, where $S$ is a del Pezzo surface, $\rho_{S}=\delta_{X}+1$, and $\delta_{T}\leq \delta_{X}$.
\end{thm} 
In Lemmas \ref{delta1} and \ref{delta}, using the numerical classes of the divisor $E_{i}$, $\hat{E}_{i}$, we find some lower-bounds for $\delta_{X}$ (in terms of $\rho_{X}-\rho_{Y}$), that we need to show Theorem \ref{ThmINTRODUZ}. We conclude this section with the proof of  Theorem \ref{singular}.\\

Section 4 contains the central part of the paper: we state and prove Theorem \ref{TEOPRINC}, which is a slightly stronger version of Theorem \ref{ThmINTRODUZ}, and some corollaries are discussed. 

Let us summarize the strategy used to prove Theorem \ref{ThmINTRODUZ}. 
To show $(1)$ we first prove that $\delta_{X}\geq 4$. Then, by Theorem \ref{prodotto} we know that $X$ is a product of a surface with another variety, and it is easy to deduce the rest of the statement.

The proof of $(2)$ is more complex and will be divided into some steps. We first analyse the simpler case in which $X$ is a product of a surface with another variety, and we get the claim similarly to $(1)$. Then we treat the case in which $X$ is not such a product, and we use some results related to the MMP for divisors in Fano variety. In particular, we apply MMP to one divisor among $E_{i}$, $\hat{E}_{i}$, for some $i=1,2$. In this way, we get a special prime divisor of $X$ that dominates $Y$. This divisor is a $\mathbb{P}^{1}$-bundle, and the images of its fibers through the conic bundle give a family of rational curves which span an extremal ray of $Y$, whose contraction is the smooth $\mathbb{P}^{1}$-fibration $\xi\colon Y\to Y^{\prime}$ required by the statement. Finally, we prove that the conic bundle has only reduced fibers, and we see that this condition implies that $Y$ and $Y^{\prime}$ are both Fano. 

The proof of $(3)$ adopts the same technique of $(2)$: we apply the MMP to a prime divisor of $X$ to get another special prime divisor of $X$ which dominates $Y$, and in the same way as before we find a smooth $\mathbb{P}^{1}$-fibration on $Y$. 

We conclude with Section 5, where applying our main Theorem we obtain some examples and related results. \\

\textit{Acknowledgements}. I wish to express my thanks to my PhD advisor Cinzia Casagrande for having introduced me to the subject. I am grateful for her constant guidance, for her several and essential suggestions, and for everything I learnt thanks to her. 

The final part of this work was written during a visit to the Institute of Mathematics of the Polish Academy of Sciences, and it was partially supported by the grant 346300 for IMPAN from the Simons Foundation and the matching 2015-2019 Polish MNiSW fund. I would to thank IMPAN for the kind hospitality. 

Finally, I am grateful to Professor St$\acute{\text{e}}$phane Druel for some helpful hints.
\section{Notation and Preliminaries}
\subsection{Notations and Conventions}
\noindent We work over the field of complex numbers. Let $X$ be a normal projective and $\mathbb{Q}$-factorial variety with arbitrary dimension $n$. We denote by $X_{reg}$ the smooth locus of $X$.

$X$ is called a \textit{Fano variety} if $-K_{X}$ admits a nonzero multiple which is Cartier and ample.\\

$\mathcal{N}_{1}(X)$ (respectively, $\mathcal{N}^{1}(X)$) is the $\mathbb{R}$-vector space of one-cycles (respectively, Cartier divisors) with real coefficients, modulo numerical equivalence. 

$\rho_{X}:= \dim{\mathcal{N}_{1}(X)}=\dim{\mathcal{N}^{1}(X)}$ is the \textit{Picard number} of $X$.

Let $C$ be a one-cycle of $X$. We denote by $[C]$ its numerical equivalence class in $\mathcal{N}_{1}(X)$, by $\mathbb{R}[C]$ the linear span of $[C]$ in $\mathcal{N}_{1}(X)$, and by $\mathbb{R}_{\geq0}[C]$ the corresponding ray. Similarly, for $D$ a $\mathbb{Q}$-Cartier divisor in $X$, we denote by $[D]$ its numerical equivalence class in $\mathcal{N}^{1}(X)$. The symbol $\equiv$ stands for numerical equivalence (for both one-cycles and $\mathbb{Q}$-Cartier divisors). If $D\subset X$ is a $\mathbb{Q}$-Cartier divisor, we define $D^{\perp}:=\{\gamma\in \mathcal{N}_{1}(X)\mid \gamma\cdot D=0 \}$.

$\NE(X)\subset \mathcal{N}_{1}(X)$ is the convex cone generated by classes of effective curves, and $\overline{\NE}(X)$ is its closure.

Let $R$ be an extremal ray of $X$. $\Lo{(R)}\subseteq X$ is the union of all curves whose class is in $R$. If $D$ is a $\mathbb{Q}$-Cartier divisor in $X$, we say that $D\cdot R> 0$ (respectively $D\cdot R=0$, etc.) if for every $\gamma\in R\setminus\{0\}$ we have $D\cdot \gamma> 0$ (respectively $D\cdot \gamma=0$, etc.).\\ 

A \textit{contraction} of $X$ is a surjective morphism $\varphi\colon X\to Y$ with connected fibers, where $Y$ is normal and projective. The push-forward of one-cycles defined by $\varphi$ is a surjective linear map $\varphi_{*}\colon \mathcal{N}_{1}(X)\to \mathcal{N}_{1}(Y)$. A contraction $\varphi\colon X\to Y$ is \textit{elementary} if $\rho_{X}-\rho_{Y}=1$. We denote by $\text{Exc}(\varphi)$ the \textit{exceptional locus} of $\varphi$, \textit{i.e.} the locus where $\varphi$ is not an isomorphism.

A contraction of $X$ is called \textit{$K_{X}$-negative} (or simply \textit{$K$-negative}) if the canonical divisor $K_{X}$ of $X$ is $\mathbb{Q}$-Cartier and $-K_{X}\cdot C>0$ for every curve $C$ contracted by $\varphi$. The \textit{relative cone} $\NE(\varphi)$ is the face of $\overline{\NE}(X)$ generated by classes of curves contracted by $\varphi$, hence $\NE(\varphi)=\overline{\NE}(X)\cap \text{Ker}(\varphi_{*})$.

If $Z\subseteq X$ is a closed subset and $i\colon Z\hookrightarrow X$ is the inclusion, we set
$\mathcal{N}_{1}(Z,X):=i_{*}(\mathcal{N}_{1}(Z))\subseteq \mathcal{N}_{1}(X)$. Equivalently, $\mathcal{N}_{1}(Z,X)$ is the linear subspace of $\mathcal{N}_{1}(X)$ spanned by classes of curves contained in $Z$. 

\subsection{Preliminaries on conic bundles} \label{conicbundle}

Let $X$ and $Y$ be smooth, projective varieties. A \textit{conic bundle} $f\colon X\to Y$ is a fiber type contraction whose fibers are isomorphic to plane conics. 


We refer the reader to \cite[$\S$4]{WIS} and to \cite{B} for equivalent definitions of conic bundles and their properties; \cite[$\S$7.1]{PAR} summarizes the relevant material on such contractions, giving in particular a survey on Mori and Mukai's results for Fano 3-folds of \cite{MM3, MM2, MM}. We set:
\begin{center} 
$\bigtriangleup_{f}:=\{y\in Y\mid f^{-1}(y) \text{ is singular}\}$.
\end{center}
We recall from \cite[$\S$1.7]{SARK} that $\bigtriangleup_{f}$ is a divisor of $Y$ that we call the \textit{discriminant divisor} of $f$, and by \cite[Proposition 1.8, (5.c)]{SARK} we have:
\begin{center} 
$\Sing{(\bigtriangleup_{f})}=\{y\in Y\mid f^{-1}(y) \text{ is non-reduced}\}$.
\end{center}
We conclude this subsection with some fundamental results on conic bundles. In particular thanks to the following theorem, we observe that Fano conic bundles can be easily characterized among contractions of smooth Fano varieties.

\begin{thm}[\cite{ANDO}, Theorem $3.1$] \label{thmWis}
Let $X$ be a smooth, projective variety, and let $f\colon X\to Y$ be a fiber type $K_{X}$-negative contraction with one-dimensional fibers. Then $Y$ is smooth and $f$ is a conic bundle.
\end{thm}
Let $X$ be a smooth Fano variety and let $f\colon X\to Y$ be a contraction. By \cite[$\S 2.5$] {CAS3}, $\NE{(Y)}$ is the linear projection of $\NE{(X)}$ from $\NE{(f)}$, thus using that every extremal ray of $\NE{(Y)}$ comes via $f_{*}$ from an extremal ray of $\NE{(X)}$, we can give a reformulation of \cite[Proposition $4.3$]{WIS}:
\begin{proposition} \label{PropWis}
Let $f\colon X\to Y$ be a Fano conic bundle. Assume that $Y$ is not Fano. Let $R$ be an extremal ray of $Y$ such that $-K_{Y}\cdot R\leq 0$. Then $\Lo{(R)}\subseteq \Sing{(\bigtriangleup_{f})}$. 
\end{proposition}

By the above proposition we obtain the following:
\begin{corollary} \label{Fano}
Let $f\colon X\to Y$ be a Fano conic bundle. If $f$ does not have non-reduced fibers, then $Y$ is  Fano.
\end{corollary}
\begin{proof} By \cite[Lemma $2.6$]{CAS3}, we know that $\NE{(Y)}$ is closed, hence by Kleiman's criterion it is enough to observe that $-K_{Y}\cdot R>0$ for every extremal ray $R$ of $Y$. This assertion follows by Proposition \ref{PropWis}. 
\end{proof}
\begin{remark} In particular, we recall that Koll$\acute{\text{a}}$r, Miyaoka, and Mori proved that if $f\colon X\to Y$ is a surjective smooth morphism between smooth varieties and where $X$ is Fano, then $Y$ is also Fano (see \cite[Corollary 2.9]{K}). Under the same assumption of smoothness of the morphism $f$, in \cite[Theorem 1.1]{FUJ} Fujino and Gongyo showed that if $X$ is weak Fano then so is $Y$.
\end{remark}
\begin{remark} In \cite{EJIRI}, Ejiri posed the following question: let $f\colon X\to Y$ be an equidimensional contraction of a Fano manifold. Assume that $f$ is not a smooth morphism, but its fibers have some mild singularities, for example semi-log canonical singularities  (see \cite[Definition 1.3]{Alexeev} for the definition of semi-log canonical singularities). Does $-K_{Y}$ have some good positivity properties? When $f\colon X\to Y$ is a Fano conic bundle, Corollary \ref{Fano} gives a positive answer to Ejiri's question. Indeed, the fibers of the conic bundle $f$ have semi-log canonical singularities if and only if they are reduced (see \cite[Example 1.4]{Alexeev}). 
\end{remark}
\subsection{Preliminaries on Special MMP's for divisors in Fano manifolds} 

We refer the reader to \cite{HU, KOLLAR} for background on the Minimal Model Program (MMP) on Mori dream spaces, and to \cite{CAS3, CAS2} for the specific properties that we will use on Fano varieties. 

By \cite{BIRKAR,HU}, we know that it is possible run a MMP for any divisor on a Fano manifold. An important remark is that when $X$ is Fano, there is always a suitable choice of a MMP where all involved extremal rays have positive intersection with the anticanonical divisor (see for instance \cite[Proposition 2.4]{CAS1}). In this last case, according to \cite{CAS1}, we call the MMP a \textit{Special MMP}.\\

We give now a technical lemma that will be needed in Section $4$. It is obtained adapting similar techniques of \cite[Lemma 2.1]{CAS2} to our specific situation. We first recall the following:

\begin{proposition}[\cite{CAS1}, Proposition $2.5$] \label{MMP}
Let $X$ be a Fano manifold and $D\subset X$ a prime divisor. Suppose that $\codim{\mathcal{N}_{1}(D,X)}=s>0$.

Then there exist $s-1$ pairwise disjoint smooth prime divisors $G_{i}\subset X$, with $i=1,\dots,s-1$, such that every $G_{i}$ is a $\mathbb{P}^{1}$-bundle with $G_{i}\cdot g_{i}=-1$, where $g_{i}\subset G_{i}$ is a fiber; moreover $D\cdot g_{i}>0$ and $[g_{i}]\notin \mathcal{N}_{1}(D,X)$. In particular $G_{i}\cap D\neq \emptyset$ and $G_{i}\neq D$. 
\end{proposition}
\begin{lemma} \label{iperpiano} In the situation of Proposition \ref{MMP}, we can assume that $\mathcal{N}_{1}(D,X)$, $\mathbb{R}[g_{1}]$, $\dots, \mathbb{R}[g_{s-1}]$ are in direct sum, so that $\mathcal{N}_{1}(D,X)\oplus \mathbb{R}[g_{1}]\oplus\dots \oplus \mathbb{R}[g_{s-1}]$ is a hyperplane in $\mathcal{N}_{1}(X)$.
\end{lemma}
\begin{proof}  
As in the proof of Proposition \ref{MMP} (see \cite[Proposition 2.5]{CAS1} for details), we consider a special MMP for $-D$:
\begin{center}$X=X_{0}\stackrel{\sigma_{0}}{\dashrightarrow} X_{1} \dashrightarrow \cdot\cdot\cdot \ \dashrightarrow X_{k-1} \stackrel{\sigma_{k-1}}{\dashrightarrow} X_{k}$
\end{center} 
where for every $i=0,\dots,k-1$, $\sigma_{i}$ is either the divisorial contraction, or the flip, of an extremal ray $R_{i}$ of $X_{i}$. 
By the proof of the same proposition, there are $i_{1},\dots, i_{s-1}$ with $0\leq i_{1}<\dots<i_{s-1}\leq k-1$ such that for every $j\in \{1,\dots,s-1\}$ we have $R_{i_{j}}\nsubseteq \mathcal{N}_{1}(D_{i_{j}}, X_{i_{j}})$, $\sigma_{i_{j}}\colon X_{i_{j}}\longrightarrow X_{i_{j+1}}$ is a divisorial contraction, $\Exc{(\sigma_{i_{j}})}$ does not intersect the loci of the birational maps $\sigma_{l}$ for $l<i$, and $G_{j}$ is the transform of $\Exc{(\sigma_{i_{j}})}$ in $X$.   

Since $G_{j}\cdot g_{j}=-1$ and $G_{i}\cdot g_{j}= 0$ for $i\neq j$, it is easy to see that the classes $[g_{1}],\dots, [g_{s-1}]$ are linearly independent.
We have to prove that $\mathcal{N}_{1}(D,X)\cap (\mathbb{R}[g_{1}] \oplus \dots \oplus \mathbb{R}[g_{s-1}])=\{0\}$. 
Assume that $\lambda_{1}[g_{1}]+\dots+\lambda_{s-1}[g_{s-1}] \in \mathcal{N}_{1}(D,X)$ for some $\lambda_{1}, \dots, \lambda_{s-1}\in \mathbb{R}$.

Set $\Gamma:= \lambda_{1} g_{1}+\dots+\lambda_{s-1} g_{s-1}$ and consider the map $X\dashrightarrow X_{i_{s-1}}$. For every $j<i_{s-1}$, since $g_{s-1}$ is contained in the open subset where $\sigma_{j}$ is an isomorphism, the transform of $\Gamma$ in $X_{i_{s-1}}$ is $\Gamma_{i_{s-1}}=\lambda_{s-1} g_{s-1}$ (for simplicity, we still denote by $g_{s-1}$ its strict transform along the MMP). By \cite[Lemma 2.1]{CAS2} we get $\lambda_{s-1}[g_{s-1}]\in N_{1}(D_{i_{s-1}},X_{i_{s-1}})$. By the construction of the MMP we have $R_{i_{s-1}}\not\subset \mathcal{N}_{1}(D_{i_{s-1}},X_{i_{s-1}})$ and $[g_{s-1}]\in R_{i_{s-1}}$, hence $[g_{s-1}]\notin \mathcal{N}_{1}(D_{i_{s-1}},X_{i_{s-1}})$, and $\lambda_{s-1}=0$. 

We can proceed in this way, repeating the same method for every $j< i_{k}$ where $k\leq s-2$, until $k=1$. We deduce that all coefficients $\lambda_{k}$ in $\Gamma$ are equal to zero, hence our statement. 
\end{proof}
We conclude this subsection by recalling the last results that will be needed in the proof of our main Theorem.
\begin{remark} [\cite{CAS1}, Remark 3.1.3, (3)] \label{intersezione} Let $X$ be a projective manifold, $G\subset X$ a smooth prime divisor which is a $\mathbb{P}^{1}$-bundle with fiber $g\subset G$, and $E\subset X$ a prime divisor with $E\cdot g>0$. Then for every irreducible curve $C\subset G$ we have $C\equiv \lambda g+\mu C^{\prime}$, where $C^{\prime}$ is an irreducible curve contained in $G\cap E$, $\lambda, \mu \in \mathbb{R}$, and $\mu\geq 0$.
\end{remark}

\begin{proposition} [\cite{CAS1}, Lemma 3.2.25] \label{smooth morphism} Let $G$ be a smooth projective variety and $\pi\colon G\to W$ a $\mathbb{P}^{1}$-bundle with fiber $g\subset G$. Moreover, let $f\colon G\to Y$ be a morphism onto a smooth projective variety $Y$, such that $\dim{f(g)}=1$. Suppose that there exists a prime divisor $A\subset Y$ such that $\mathcal{N}_{1}(A,Y)\subsetneq \mathcal{N}_{1}(Y)$ and $f^{*}(A)\cdot g> 0$. Then there exists a commutative diagram: 
$$
\begin{array}{ccc}
G & \stackrel{f}{\longrightarrow} & Y \\
\pi \downarrow & & \downarrow \xi \\
W & \longrightarrow & Y^{\prime}
\end{array}
$$ 
where $Y^{\prime}$ is smooth and $\xi$ is a smooth $\mathbb{P}^{1}$-fibration.
\end{proposition}
The next result concerns contractions of smooth, projective, Fano varieties which are products of Fano manifolds. This is probably well-known, we include a proof for lack of references.
\begin{lemma} \label{product_contraction} Let $X_{1}$, $X_{2}$ be smooth, projective, Fano varieties, and $X= X_{1}\times X_{2}$. Let $f\colon X\to Y$ be a contraction. Then $Y\cong Y_{1}\times Y_{2}$, and $f=(f_{1},f_{2})$, where $f_{1}\colon X_{1}\to Y_{1}$, $f_{2}\colon X_{2}\to Y_{2}$ are contractions.
\end{lemma}
\begin{proof}
	Take $D=f^{*}(A)$, where $A$ is an ample divisor of $Y$. Let us consider the two natural projections $\pi_{1}\colon X\to X_{1}$, $\pi_{2}\colon X\to X_{2}$. Being $X$ Fano, one has $\mathcal{N}^{1}(X)\cong \pi_{1}^{*} \mathcal{N}^{1}(X_{1})\oplus \pi_{2}^{*} \mathcal{N}^{1}(X_{2})$ (see for instance \cite[Ex. III, 12.6]{HART}), then $D=\pi_{1}^{*}D_{1}+\pi_{2}^{*}D_{2}$, where $D_{i}$ is a divisor of $X_{i}$ for $i=1,2$.
	
	Let us consider $D_{1}$, and let $C\subset X_{1}$ be an irreducible curve. Set $C^{\prime}=C\times \{x_{2}\}$, where $x_{2}\in X_{2}$ is a point. Using the projection formula we get $\pi_{2}^{*} D_{2}\cdot C^{\prime}=D_{2}\cdot (\pi_{2})_{*} C^{\prime}=0$, hence:
	\begin{center}
		$D_{1}\cdot C=D_{1}\cdot {\pi_{1}}_{*}(C^{\prime})=(\pi_{1}^{*}D_{1}+\pi_{2}^{*}D_{2})\cdot C^{\prime}=D\cdot C^{\prime}\geq 0$,
	\end{center} where the last inequality holds because $D$ is a nef divisor of X. Thus $D_{1}$ is nef on $X_{1}$ Fano, hence $D_{1}$ is semiample, and we obtain a contraction $f_{1}\colon X_{1}\to Y_{1}$ such that $D_{1}={f_{1}}^{*}(A_{1})$, with $A_{1}$ ample on $Y_{1}$. By the same argument, we get a contraction $f_{2}\colon X_{2}\to Y_{2}$ such that $D_{2}={f_{2}}^{*}(A_{2})$, with $A_{2}$ ample on $Y_{2}$. 
	
	Set $g:=(f_{1}, f_{2})\colon X\to Y_{1}\times Y_{2}$, let $p_{i}\colon Y_{1}\times Y_{2}\to Y_{i}$ the projection for $i=1,2$, and $A=p^{*}(A_{1})+q^{*}(A_{2})$, which is an ample divisor of $Y_{1}\times Y_{2}$. Then:
	\begin{center} $g^{*}(A)=g^{*}{p_{1}}^{*}(A_{1})+g^{*}{p_{2}}^{*}(A_{2})=\pi_{1}^{*}f_{1}^{*}(A_{1})+\pi_{2}^{*}f_{2}^{*}(A_{2})=\pi_{1}^{*}D_{1}+\pi_{2}^{*}D_{2}=D$, 
	\end{center}
	which implies that $Y\cong Y_{1}\times Y_{2}$ and that $g=f$ under this isomorphism, by \cite[Proposition 1.14]{DEB}.
\end{proof}
\section{Fiber type $K$-negative contraction with one-dimensional fibers} \label{terzocapitolo}
\noindent The first results of this section, Proposition \ref{TEO1} and \ref{smoothcase}, can be viewed as a generalization to higher dimension of Mori and Mukai's results on conic bundles on Fano 3-folds in \cite{MM}, in particular \cite[Proposition 4.9]{MM}.

We first study what happens when the variety has some mild singularities, and then we restrict to the smooth case.
To prove our first proposition in the singular case, we need the following:
\begin{thm}[\cite{NOCE}, Theorem 1.2] \label{NOCE} Let $X$ be a projective, locally factorial variety with canonical singularities and with at most finitely many non-terminal points. Let $\varphi\colon X\to Y$ be an elementary birational $K$-negative contraction whose fibers are at most one-dimensional. Then:
\begin{enumerate}[(1)]
\item every non-trivial fiber of $\varphi$ is irreducible, has no multiple one dimensional components and its reduced structure is isomorphic to $\mathbb{P}^{1}$.
Moreover the general non-trivial fiber is smooth, i.e. it is isomorphic to $\mathbb{P}^{1}$ as scheme;
\item the contraction $\varphi$ is divisorial, and denoting by $E=\Exc{(\varphi)}$, one has $\dim{\varphi(E)}=n-2$. Moreover $K_{X}=\varphi^{*}(K_{Y})+E$;
\item Y has canonical singularities and at most finitely many non terminal points;
\item let $C$ be an irreducible curve of $X$ such that $[C]\in \NE{(\varphi)}$. We have that $K_{X}\cdot C=E\cdot C=-1$.
\end{enumerate}
\end{thm}  
\begin{remark} The crucial statement in the theorem above is that $\varphi$ is divisorial. In \cite[Example 1.11]{NOCE}, the author shows that the assumption on the non-terminal locus cannot be weakened.
\end{remark}
\begin{lemma} \label{locally}
In the context of Theorem \ref{NOCE}, $Y$ is a locally factorial variety.
\end{lemma}
\begin{proof} This is a standard property, we give the proof for the reader's convenience. The strategy adopted is quite similar to the proof of \cite[Proposition 7.44]{DEB}.

Let $D\subset Y$ be a prime Weil divisor, and $\tilde{D}\subset X$ its transform, that is a Cartier divisor because $X$ is locally factorial. Set $E:=\Exc{(\varphi)}$. By Theorem \ref{NOCE} $(4)$, if $C\subset X$ is an irreducible curve such that $[C]\in \NE{(\varphi)}$, then $E\cdot C=-1$. Thus $(\tilde{D}+\lambda E)\cdot C=0$, where $\lambda=\tilde{D}\cdot C \in \mathbb{Z}$. Hence $\tilde{D}+\lambda E$ is a Cartier divisor and by \cite[Theorem $7.39$ (c)]{DEB} we get $\tilde{D}+\lambda E=\varphi^{*}(B)$, where $B$ is an effective Cartier divisor of $Y$. Then, $\text{Supp} (B)= \text{Supp} (D)$ in $Y$, so that $B=m D$ for some $m\in \mathbb{Z}_{\geq 1}$. By restricting to the open subset where $\varphi$ is an isomorphism, we get $m=1$, so that $D$ is a Cartier divisor of $Y$.
\end{proof}

Now we are ready to give our first result, that will be essential to investigate conic bundles. 
\begin{proposition} \label{TEO1}
Let $X$ be a locally factorial, projective variety with canonical singularities and with at most finitely many non-terminal points.

Let $f\colon X\to Y$ be a fiber type $K$-negative contraction such that every fiber has dimension one. Set $dim \ {\NE{(f)}}=\rho_{X}-\rho_{Y}=r$. Then:
\begin{enumerate} [(1)]
\item $f$ has the following factorization: 
\begin{center}$X \stackrel{f_{1}}{\longrightarrow} X_{1}\stackrel{f_{2}}{\longrightarrow}  X_{2}\stackrel{}{\longrightarrow} \cdot\cdot\cdot \stackrel{}{\longrightarrow} X_{r-2}\stackrel{f_{r-1}}{\longrightarrow} X_{r-1}\stackrel{g}{\longrightarrow} Y$ 
\end{center} 
where $f_{i}$ is an elementary $K$-negative divisorial contraction, $X_{i}$ is locally factorial, with canonical singularities and at most finitely many non-terminal points for every $i=1,\dots,r-1$, and $g$ is an elementary fiber type $K$-negative contraction with one-dimensional fibers.
\item There are $A_{1},\dots,A_{r-1}$ prime divisors of $Y$ and $r-1$ pairs of prime divisors $E_{i}, \hat{E_{i}}$ of $X$ such that $f^{*}(A_{i})=E_{i}+\hat{E_{i}}$ and $A_{i}\cap A_{j}=\emptyset$, for every $i\neq j$, $i,j=1,\dots,r-1$. Moreover, every fiber of $g$ over $A_{1}\cup\dots\cup A_{r-1}$ is irreducible and generically reduced, and the general fiber of $f$ is numerically equivalent to $e_{i}+\hat{e}_{i}$ where $e_{i}, \hat{e}_{i}$ are irreducible components of fibers of $f$ such that $e_{i}\subset E_{i}$, $\hat{e}_{i}\subset \hat{E_{i}}$, $E_{i}\cdot e_{i}=\hat{E_{i}}\cdot\hat{e}_{i}=-1$ and  $\hat {E_{i}}\cdot e_{i}=E_{i}\cdot\hat{e}_{i}=1$. 
\item If $r>1$, $\NE{(f)}$ has exactly $2(r-1)$ extremal rays, generated by $[e_{i}]$, $[\hat{e}_{i}]$, for $i=1,\dots,r-1$.
\end{enumerate}
\end{proposition}


\begin{proof} We prove $(1)$ and $(2)$ by induction on $\rho_{X}-\rho_{Y}=r$. The case $r=1$ is trivial. Assume that $r>1$, and that the statement holds when the dimension of the relative cone is less than $r$.

Let $R$ be an extremal ray of $\NE{(f)}$ and let $f_{1}\colon X\to X_{1}$ be its contraction. Being $\NE{(f_{1})}\subset \NE{(f)}$, then there exists a morphism $g_{1}\colon X_{1}\to Y$ such that $f=g_{1} \circ f_{1}$. One has that $f_{1}$ is $K$-negative, because so is $f$.
$$\xymatrix{{X}\ar[r]_{f_{1}}\ar@/^1pc/[rr]^f&X_{1}\ar[r]_{g_{1}}&Y
}$$
We observe that $f_{1}$ is birational. Indeed, suppose by contradiction that $f_{1}$ is a fiber type contraction. In this case, we have that $g_{1}$ is a finite morphism: indeed, if there exists an irreducible curve $C\subset X_{1}$ such that $g_{1}(C)=p$, then $f^{-1}(p)$ has dimension greater than one, against our assumption. Hence $\NE{(f)}=\NE{(f_{1})}$ and $f$ is an elementary contraction, while we are assuming that $r>1$. Thus $f_{1}$ is birational, and it has one-dimensional fibers, so that it is divisorial by Theorem \ref{NOCE} $(2)$, and $g_{1}$ is a fiber type contraction with one-dimensional fibers.

Applying Theorem \ref{NOCE} $(3)$ and Lemma \ref{locally}, one has that $X_{1}$ has canonical singularities with at most finitely many non-terminal points, and it is locally factorial. Set $A_{1}^{\prime}:=f_{1}(\Exc{(f_{1})})$. 

We observe that $g_{1}$ is a finite morphism on $A_{1}^{\prime}$: if there is a curve $C\subset A_{1}^{\prime}$ contracted by $g_{1}$ to a point $p\in Y$, $f^{-1}(p)$ would have dimension greater than one, against our hypothesis. Hence $A_{1}:=g_{1}(A_{1}^{\prime})$ is an irreducible divisor of $Y$.

We show that $g_{1}$ is $K$-negative. To this end, since the fibers of $g_{1}$ are one-dimensional, it is enough to show that $-K_{X_{1}}\cdot \Gamma>0$ for every irreducible curve $\Gamma\subset X_{1}$ such that $g_{1}(\Gamma)=\{pt\}$. Let $C\subset X_{1}$ be such a curve. Then $C\not\subset A_{1}^{\prime}$ and consider $\tilde{C}$ the transform of $C$ in $X$, so that $(f_{1})_{*}(\tilde{C})=C$. By Theorem \ref{NOCE} $(2)$, $K_{X}=f_{1}^{*}(K_{X_{1}})+E$ where $E=\Exc{(f_{1})}$. Then:
\begin{center} 
$-K_{X_{1}} \cdot C= -K_{X_{1}}\cdot(f_{1})_{*} (\tilde{C})= f_{1}^{*}(-K_{X_{1}}) \cdot \tilde{C}= (-K_{X}+E)\cdot\tilde{C}>0$.
\end{center}
Thus $g_{1}$ is $K$-negative, and applying the induction assumption to $g_{1}$, we get $(1)$. Moreover, still by the induction assumptions, we get $r-2$ prime divisors $A_{2}, \dots, A_{r-1}$ of $Y$, pairwise disjoint, such that $g_{1}^{*}(A_{i})=F_{i}+\hat{F}_{i}$, where $F_{i}, \hat{F}_{i}$ are prime divisors of $X_{1}$, for every $i=2,\dots,r-1$.  

We show that the fibers of $f$ have at most two components. Denoting by $e$ the general fiber, one has $-K_{X}\cdot e= 2$ because the normal bundle of $e$ in $X$ is trivial, and $e\cong \mathbb{P}^{1}$. Since each irreducible component of every fiber of $f$ has anticanonical degree at least 1, there are no more than two components for every fiber.

We notice that the fibers of $g_{1}$ that meet $A_{1}^{\prime}$ are irreducible, generically reduced, and intersect set-theoretically $A_{1}^{\prime}$ in one point. If a fiber of $g_{1}$ intersects $A_{1}^{\prime}$ in more than one point, then $f$ would have a fiber with more than two components. The same holds if a reducible fiber of $g_{1}$ meets $A_{1}^{\prime}$. 

Since all fibers of $g_{1}$ over $A_{2}\cup \cdots \cup A_{r-1}$ are reducible, we deduce that $A_{1}\cap A_{i}=\emptyset$ for every $i=2,\dots,r-1$. Our situation is like in Figure \ref{fig:fig1}.

\begin{figure}[ht]
	\centering
		\includegraphics[scale=0.9]{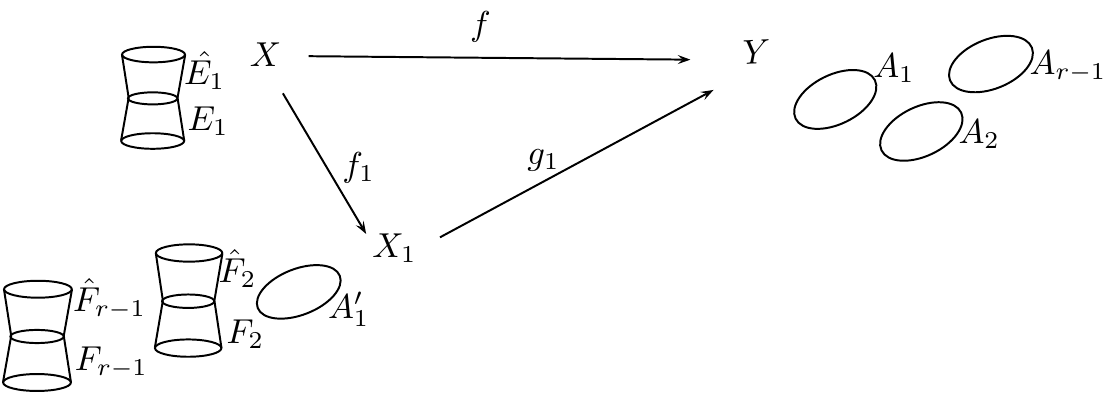}
	\caption{Divisors in the proof of Proposition \ref{TEO1}.}
	\label{fig:fig1}
\end{figure}

In particular $F_{i}, \hat{F}_{i}\subset X_{1}\setminus f_{1}(\Exc{(f_{1})})$; let $E_{i}, \hat{E}_{i} \subset X$ be their transforms, so that $f^{*}(A_{i})=E_{i}+\hat{E}_{i}$, for every $i=2,\dots,r-1$.

By what we have already shown, for every $q\in A_{r-1}$, $f^{-1}(q)$ will have two irreducible components: one contracted by $f_{1}$ and another one that is the transform of $g_{1}^{-1}(q)$ in $X$. Then $f^{*}(A_{1})=E_{1}+\hat{E}_{1}$ where $E_{1}$ is the exceptional locus of $f_{1}$, and $\hat{E}_{1}$ is the transform in $X$ of the irreducible divisor $g_{1}^{-1}(A_{1})$ of $X_{1}$. Hence every fiber of $f$ over $A_{1}$ will be numerically equivalent to $e_{1}+\hat{e}_{1}$ where $e_{1}$ and $\hat{e}_{1}$ are irreducible components of the fibers such that $e_{1}\subset E_{1}$ and $\hat{e}_{1}\subset \hat{E}_{1}$.

Finally by Theorem \ref{NOCE} $(4)$, we have $\hat{E}_{1} \cdot \hat{e}_{1}=-1$. This easily yields the intersections required by the statement.\\

We show $(3)$. By our assumption, $r>1$. Being $\NE{(f)}$ generated by the components of the fibers of $f$, it is spanned by $\{[e_{i}], [\hat{e}_{i}]\}_{i=1,\dots,r-1}$. We need to show that every $[e_{i}]$ and $[\hat{e}_{i}]$ spans an extremal ray. To deduce this, we examine the fibers $e_{i}\cup \hat{e}_{i}$. Since $\hat{E_{i}}\cdot \hat{e}_{i}=-1$, there exists an extremal ray $R$ in $\NE{(f)}$ such that $\hat{E_{i}}\cdot R< 0$, thus $\Lo{(R)}\subset \hat{E_{i}}$. But $[\hat{e}_{i}]$ is the unique numerical equivalence class of curves of $\NE{(f)}$ contained in $\hat{E_{i}}$, hence $R=\mathbb{R}_{\geq0}[\hat{e_{i}}]$. Similarly, using that $E_{i}\cdot e_{i}< 0$, $R=\mathbb{R}_{\geq0}[e_{i}]$ is an extremal ray of $\NE{(f)}$.
\end{proof}

The following proposition tells us what happens in the smooth case. It is probably well-known to experts in the field, but we could not find a suitable reference. 
\begin{proposition} \label{smoothcase} Notation as in Proposition \ref{TEO1}. Assume that $X$ is smooth. Then we have:
\begin{enumerate}[(1)]
\item $f$ is a $conic \ bundle$, $g$ is an elementary conic bundle, every $A_{i}$ is smooth, and $f_{i}$ is the blow-up of the manifold $X_{i}$ along a smooth subvariety of codimension 2 that is isomorphic to $A_{i}$, for every $i=1,\dots,r-1$. 
\item Let $\bigtriangleup_{f}=\{y \in Y\mid f^{-1}(y) \text{ is not a smooth conic}\}$ be the discriminant divisor of $f$. One has $\bigtriangleup_{f}=A_{1}\sqcup\dots\sqcup A_{r-1}\sqcup \bigtriangleup_{g}$, and each $A_{i}$ is a smooth connected component of $\bigtriangleup_{f}$, for every $i=1,\dots,r-1$.
Moreover, $f$ has reduced fibers over $A_{1}\cup\dots\cup A_{r-1}$.
\end{enumerate}
\end{proposition}
\begin{proof} We first show $(1)$. By Theorem \ref{thmWis} it follows that $f$ is a conic bundle, $g$ is an elementary conic bundle, and by \cite[Theorem 2.3]{ANDO} every $f_{i}\colon X_{i-1}\to X_{i}$ is the blow-up of the manifold $X_{i}$ along a smooth subvariety of codimension 2. Let us denote by $A_{i}^{\prime}$ the centers of the blow-ups $f_{i}$, and we still denote by $A_{i}^{\prime}$ their image in $X_{r-1}$. In the proof of Proposition \ref{TEO1} we have already shown that $g$ is finite on $A_{i}^{\prime}$, and that the fibers of $g$ over $A_{i}$ are irreducible and intersect $A_{i}^{\prime}$ in only one point, for every $i=1,\dots,r-1$. 

We show that $g_{\mid_{A_{i}^{\prime}}}\colon A_{i}^{\prime}\to A_{i}$ is an isomorphism. For this purpose, we prove that the intersection between a fiber of $g_{\mid_{A_{i}^{\prime}}}$ and $A_{i}^{\prime}$ is transversal. Let $p \in A_{i}$. Set $g_{i}:= g\circ f_{r-1}\circ \dots \circ f_{i+1}\colon X_{i}\to Y$, where $i\in \{1,\dots,r-1\}$, and $g_{r-1}=g$. 
Let $\Gamma$ be the transform in $X_{i-1}$ of the fiber $g^{-1}(p)$. Since $g^{-1}(p)$ does not intersect the indeterminacy locus of $f_{r-1}\circ \dots\circ f_{i+1}$, then $g^{-1}(p)\cong g_{i}^{-1}(p)$.
We get:
\begin{center}
$\Exc{(f_{i})}\cdot \Gamma=(f_{i}^{*}(-K_{X_{i}})+K_{{X}_{i-1}})\cdot \Gamma=-K_{X_{i}}\cdot g_{i}^{-1}(p)-(-K_{X_{i-1}}\cdot \Gamma)<2$.
\end{center}

Hence $A_{i}^{\prime}$ is a section of $g$, and $A_{i}^{\prime}\cong A_{i}$ for every $i=1,\dots,r-1$. \\
 
Now we prove $(2)$. The inclusion $A_{1}\sqcup \dots \sqcup A_{r-1}\sqcup \bigtriangleup_{g} \subseteq \bigtriangleup_{f}$ is a simple consequence of some facts that we proved in Proposition \ref{TEO1}. In particular all fibers over each $A_{i}$ are reducible and every singular fiber of $g$ does not meet the indeterminacy locus of $f_{r-1}\circ f_{r-2}\circ \dots \circ f_{1}$, hence it is isomorphic to a singular fiber of $f$. On the other hand, if we take $y \in \bigtriangleup_{f}$ such that $y \not \in A_{1}\sqcup\cdots \sqcup A_{r-1}$ one has that $f^{-1}(y)$ is isomorphic to a singular fiber of $g$, hence $y \in \bigtriangleup_{g}$ and the equality holds. Moreover, being all fibers over $A_{i}$ reducible, by \cite[Proposition 1.8]{SARK}, we deduce that every $A_{i}$ is a smooth component of $\bigtriangleup_{f}$.  
\end{proof}
\begin{remark} In the setting of Proposition \ref{smoothcase}, we observe that the factorization of $f$, and thus the elementary fiber type contraction $g$ of the factorization, are not unique. Indeed, the factorization of $f$ depends by the choice of $r-1$ extremal rays of $\NE{(f)}$, one among $\mathbb{R}_{\geq 0}[e_{i}]$ and $\mathbb{R}_{\geq 0}[\hat{e}_{i}]$, for every $i=1,\dots,r-1$. 
\end{remark}
\begin{remark} \label{uniquelydet} In the setting of Proposition \ref{smoothcase}, let $\bigtriangleup_{0}$ be an irreducible component of $\bigtriangleup_{f}$. Then $f^{*}(\bigtriangleup_{0})$ is reducible if and only if $\bigtriangleup_{0}=A_{i}$ for some $i\in \{1,\dots,r-1\}$. In fact, $f^{*}(A_{i})=E_{i}+\hat{E}_{i}$ is reducible. On the other hand, if $\bigtriangleup_{0}\neq A_{i}$, then $\bigtriangleup_{0}\subseteq \bigtriangleup_{g}$ and $f^{*}(\bigtriangleup_{0})$ is irreducible by \cite[Proposition 4.8, (1)]{MM}. Thus $\bigtriangleup_{g}\subseteq \bigtriangleup_{f}$ is uniquely determined by $f$, and also the smoothness of $g$ that is equivalent to $\bigtriangleup_{f}=A_{1}\cup\dots\cup A_{r-1}$ is independent on the factorization of $f$.
\end{remark}
%


By Remark \ref{uniquelydet} and Corollary \ref{Fano} one gets the following:
\begin{corollary} \label{onlyred} Let $f\colon X\to Y$ be a Fano conic bundle, and set $r:=\rho_{X}-\rho_{Y}$. Let $g$ be as in Proposition \ref{smoothcase}. If $g$ is smooth, then $f$ does not have non-reduced fibers and $Y$ is Fano.
\end{corollary}
Now, using the $r-1$ pairs of divisors $E_{i}$, $\hat{E}_{i}$ as in Proposition \ref{TEO1}, we find a bound for the dimension of the relative cone $\NE{(f)}$. In particular, we are going to prove Theorem \ref{singular}. To this end, we need the two following lemmas. 
\begin{lemma} \label{div. ind} Notation as in Proposition \ref{TEO1}, and assume that $r\geq 2$. The $r$ numerical equivalence classes $[E_{1}],[E_{2}],\dots,[{E_{r-1}}],[\hat{E_{1}}]$ are linearly independent in $\mathcal{N}^{1}(X)$.
\end{lemma}
\begin{proof}Suppose that $\hat{a}_{1} \hat{E_{1}}+\sum_{i=1}^{r-1} a_{i} E_{i}\equiv 0$, for some $\hat{a}_{1}, a_{i} \in \mathbb{R}$. The intersections with $e_{i}$, for every $i=2,\dots,r-1$, give us $a_{i}=0$, then $\hat{a}_{1} \hat{E_{1}}+ a_{1} E_{1}\equiv 0$. Intersecting with $e_{1}$ one has $\hat{a}_{1}=a_{1}$ so that $a_{1}(\hat{E_{1}}+E_{1})\equiv 0$. Since $\hat{E_{1}}+E_{1}$ is an effective divisor, it cannot be numerically equivalent to zero, hence $a_{1}=0$.
\end{proof}
\begin{lemma} \label{delta1} Let $f\colon X\to Y$ be as in Proposition \ref{TEO1}, and set $r:=\rho_{X}-\rho_{Y}$. If $r\geq 3$, then $\delta_{X}\geq r-1$.
\end{lemma}
\begin{proof} Let $E_{i},\hat{E}_{i}\subset X$ be the $r-1$ pairs of prime divisors as in Proposition \ref{TEO1}. They satisfy $E_{r-1}\cap(E_{1}\cup\cdots\cup E_{r-2}\cup \hat{E}_{1})=\emptyset$, hence we find that $\mathcal{N}_{1}(E_{r-1},X)\subseteq {E}_{1}^{\perp}\cap\dots\cap {E}_{r-2}^{\perp}\cap \hat{E}_{1}^{\perp}$. Lemma \ref{div. ind} yields that $\codim {\mathcal{N}_{1}(E_{r-1},X)}\geq r-1$, thus $\delta_{X}\geq r-1$. 
\end{proof}
\begin{proof}[Proof of Theorem \ref{singular}] Set $r:=\rho_{X}-\rho_{Y}$. Suppose that $r\geq 3$, otherwise there is nothing to prove. Using Lemma \ref{delta1}, we deduce that $\delta_{X}\geq r-1$. We observe that $X$ is Gorenstein. To this end, since $K_{X}$ is a Cartier divisor, it is enough to notice that $X$ is Cohen-Macaulay. Using \cite[Corollary 5.24]{KOLLAR} we know that $X$ has rational singularities and by \cite[Theorem 5.10]{KOLLAR}, it follows that these singularities are Cohen-Macaulay. By Theorem \ref{prodotto}, $r-1\leq 8$, hence $r\leq 9$.
\end{proof}
\section{Main Theorem on non elementary Fano conic bundles}
\noindent This section contains the central part of the paper, where we investigate \textit{non-elementary} Fano conic bundles.
 
Given a non-elementary conic bundle $f\colon X\to Y$ we can take a factorization as in Propositions \ref{TEO1} and \ref{smoothcase}. When the elementary conic bundle of the factorization of $f$ is singular, we improve Lemma \ref{delta1} in the following way:

\begin{lemma} \label{delta} Let $f\colon X\to Y$ be a Fano conic bundle. Let consider a factorization of $f$ as in Proposition \ref{TEO1} (1), and denote by $g$ the elementary conic bundle of this factorization. Set $r:=\rho_{X}-\rho_{Y}$. If $g$ is singular and $r\geq 2$, then $\delta_{X}\geq r$.
\end{lemma}
\begin{proof} By our assumption, $\bigtriangleup_{g}\neq\emptyset$. Let us consider its inverse image $\tilde{\bigtriangleup}_{g}$ in $X$. Take an irreducible component of this divisor that we call $\tilde{{\bigtriangleup}^{\prime}}_{g}$. Note that $\tilde{{\bigtriangleup}^{\prime}}_{g}\cap(E_{1}\cup\cdots\cup E_{r-1}\cup \hat{E}_{1})=\emptyset$, thus $\mathcal{N}_{1}(\tilde{{\bigtriangleup}^{\prime}}_{g},X)\subseteq {E}_{1}^{\perp}\cap\dots\cap {E}_{r-1}^{\perp}\cap \hat{E}_{1}^{\perp}$. Using Lemma \ref{div. ind} we get the statement. 
\end{proof}
\begin{thm} \label{TEOPRINC} Let $f\colon X\to Y$ be a Fano conic bundle. Then $\rho_{X}-\rho_{Y}\leq 8$. Moreover:
\begin{enumerate}[(1)]
\item if $\rho_{X}-\rho_{Y}\geq 4$, then $X\cong S\times T$, where $S$ is a del Pezzo surface, $T$ is an $(n-2)$-dimensional Fano manifold, $Y\cong \mathbb{P}^{1}\times T$ so that $Y$ is Fano, and $f$ is induced by a conic bundle $S\to \mathbb{P}^{1}$.
\end{enumerate}
Let us denote by $g\colon X_{r-1}\to Y$ the elementary conic bundle in a factorization of $f$ as in Propositions \ref{TEO1} and \ref{smoothcase}.
\begin{enumerate}[(2)]
\item If $\rho_{X}-\rho_{Y}=3$, then $g$ is smooth, $Y$ is a Fano variety, and there exists a smooth $\mathbb{P}^{1}$-fibration $\xi\colon Y\to Y^{\prime}$ where $Y^{\prime}$ is smooth and Fano. 
\end{enumerate}
\begin{enumerate}[(3)]
\item If $\rho_{X}-\rho_{Y}=2$, and $g$ is singular, then there exists a smooth $\mathbb{P}^{1}$-fibration $\xi\colon Y\to Y^{\prime}$ where $Y^{\prime}$ is smooth.
\end{enumerate}
\end{thm}
\begin{proof}[\textbf{Proof of (1)}] Set $r:=\rho_{X}-\rho_{Y}$. Let us first show that $\delta_{X}\geq 4$. 

If $r>4$, this inequality follows by Lemma \ref{delta1}. Assume that $r=4$. We have three pairs of prime divisors $E_{i},\hat{E}_{i}\subset X$, for $i=1,2,3$, as in Proposition \ref{TEO1}. 
Suppose that there exist two pairs such that the numerical classes of the four divisors of the pairs are linearly independent in $\mathcal{N}^{1}(X)$. We can assume that they are $E_{1}, \hat{E}_{1}, E_{2}, \hat{E}_{2}$. Then $\mathcal{N}_{1}(E_{3},X)\subseteq {E}_{1}^{\perp} \cap {\hat{E}}_{1}^{\perp} \cap {E}_{2}^{\perp} \cap {\hat{E}}_{2}^{\perp}$ and it follows that $\delta_{X}\geq 4$.

Let assume now that for every two pairs $E_{i},\hat{E}_{i}$, the numerical classes of the four divisors are linearly dependent in $\mathcal{N}^{1}(X)$. 
This means, for instance, that $E_{1}\equiv \hat{a} \hat{E}_{1}+b E_{2}+\hat{b} E_{2}$ for some $\hat{a}, b, \hat{b}\in \mathbb{R}$. The intersection with the fibers $e_{1}$ and $e_{2}$ gives $E_{1}+\hat{E}_{1}\equiv b (E_{2}+\hat{E}_{2})$. In the same way, from the linear dependence of the numerical classes $[E_{1}], [\hat{E}_{1}], [E_{3}], [\hat{E}_{3}]$, we get the relation $E_{1}+\hat{E}_{1}\equiv c (E_{3}+\hat{E}_{3})$ where $c\in \mathbb{R}$. 

Let $A_{1}$, $A_{2}$, $A_{3}$ be as in Proposition \ref{TEO1}, so that $f^{*}(A_{i})=E_{i}+\hat{E}_{i}$. Hence $f^{*}(A_{1})\equiv b f^{*}(A_{2})\equiv c f^{*}(A_{3})$ and being $f^{*}\colon \mathcal{N}^{1}(Y)\to \mathcal{N}^{1}(X)$ injective, one has $A_{1}\equiv b A_{2}\equiv c A_{3}$. Moreover, by Proposition \ref{TEO1} $(2)$, we know that the divisors $A_{i}$ are pairwise disjoint, so these three divisors are nef and cut a facet of $\NE{(Y)}$, whose contraction $\Phi\colon Y\to Z$ sends $A_{1}, A_{2}, A_{3}$ to points (see \cite[Lemma 2.6]{CAS3}), so $\dim{Z}=1$. Since $X$ is Fano, $X$ is rationally connected, so that $Z\cong \mathbb{P}^{1}$. Hence we get a contraction $X\stackrel{f}{\to} Y \stackrel{\Phi}{\to} \mathbb{P}^{1}$. The general fiber $G$ of $\Phi \circ f$ is an irreducible divisor of $X$ disjoint from $E_{i}$ and $\hat{E}_{i}$ for every $i=1,2,3$, thus $\mathcal{N}_{1}(G,X)\subseteq E_{1}^{\perp}\cap {\hat{E}}_{1}^{\perp}\cap {E}_{2}^{\perp}\cap E_{3}^{\perp}$. By Lemma \ref{div. ind}, we find again that $\delta_{X}\geq 4$. \\ 

Since $\delta_{X}\geq 4$, by Theorem \ref{prodotto} we know that $X\cong S\times T$, where $S$ is a del Pezzo surface and $T$ is an $(n-2)$-dimensional Fano manifold. By Lemma \ref{product_contraction} we have that $Y\cong S_{1}\times T_{1}$, where $S_{1}$, $T_{1}$ are smooth projective varieties, and the Fano conic bundle $f$ takes the following form: $f=(h_{1},h_{2})$ where $h_{1}\colon S\to S_{1}$, $h_{2}\colon T\to T_{1}$. We get a partition of $Y$ in two subsets:

$Y_{1}=\{(p,q)\in Y\mid f^{-1}(p,q)=F_{1}\times\{point\} \ \text{where} \ \dim {F_{1}}=1\}$ and 

$Y_{2}=\{(p,q)\in Y\mid f^{-1}(p,q)=\{point\}\times F_{2} \ \text{where} \ \dim {F_{2}}=1\}$.

By the upper-semicontinuity of fiber dimension there are two possibilities: 

$(a)$ $Y_{1}=Y$ and $Y_{2}=\emptyset$; 

$(b)$ $Y_{1}=\emptyset$ and $Y_{2}=Y$. 

Let us assume that $(a)$ holds. Then $h_{1}$ is a conic bundle on $S$, $\dim{S_{1}}=1$, $S_{1}\cong \mathbb{P}^{1}$, and $h_{2}$ is the identity on $T$. Hence $Y\cong\mathbb{P}^1\times T$ and the statement follows.

Let us now suppose that $(b)$ holds. We have that $h_{1}$ is the identity on $S$, $h_{2}$ is a Fano conic bundle on $T$, and $\rho_{T}-\rho_{T_{1}}=\rho_{X}-\rho_{Y}\geq 4$. 

By induction on the dimension, we deduce that $T\cong S_{2}\times T_{2}$, where $S_{2}$ is a del Pezzo surface, $T_{2}$ is an $(n-4)$-dimensional Fano manifold, $T_{1}\cong \mathbb{P}^{1}\times T_{2}$, and $h_{2}:T\to T_{1}$ is induced by a conic bundle $S_{2}\to \mathbb{P}^{1}$. 

We can conclude that $X\cong S_{2}\times S\times T_{2}$, $Y\cong \mathbb{P}^{1}\times S\times T_{2}$, and $f$ is induced by the conic bundle $S_{2}\to \mathbb{P}^{1}$, hence we get the statement. \\ 

\noindent \textbf{\textit{Proof of (2)}}. The proof will be achieved in some steps.\\

\noindent \textit{\textbf{Step 1}}: The statement holds when $X\cong S\times T$, where $S$ is a del Pezzo surface and $T$ is a Fano manifold.

\begin{proof}[Proof of Step 1] Keeping the notation used in the proof of $(1)$, we get the same two cases that we call again $(a)$ and $(b)$.

If $(a)$ holds, we deduce as before that $Y\cong \mathbb{P}^{1}\times T$, and that $f$ is induced by a conic bundle $S\to \mathbb{P}^{1}$. In particular, we get $(2)$.

If $(b)$ holds, we have $Y\cong S\times T_{1}$, and $f=(id_{S}, h_{2})$, where $h_{2}:T\to T_{1}$ is a Fano conic bundle with $\rho_{T}-\rho_{T_{1}}=\rho_{X}-\rho_{Y}=3$.
We can proceed by induction on the dimension. We apply Proposition \ref{smoothcase} $(1)$ to $h_{2}$, and we denote by $g^{\prime}$ the elementary conic bundle in the factorization of $h_{2}$. By induction, we find that $g^{\prime}$ is smooth, and being $g=(id_{S}, g^{\prime})$, $g$ is also smooth. Moreover, still by induction, there exists a smooth $\mathbb{P}^{1}$-fibration $\xi^{\prime}\colon T_{1}\to Y_{1}$, where $Y_{1}$ is smooth and Fano, so that $\xi:=(id_{S},\xi^{\prime})\colon Y\to S\times Y_{1}$ is a smooth $\mathbb{P}^{1}$-fibration onto a variety that is smooth and Fano, and this shows $(2)$. 
\end{proof}

From now on, we suppose that $X\ncong S\times T$. In particular, by Theorem \ref{prodotto} we know that $\delta_{X}\leq 3$, and by Lemma \ref{delta1} one has that $\delta_{X}\geq 2$. Thus there are only two possibilities: $\delta_{X}=2$ or $\delta_{X}=3$. 

Recall also that $\rho_{X}-\rho_{Y}=3$, hence by Proposition \ref{TEO1} we have two pairs $E_{1}, \hat{E}_{1}$ and $E_{2}, \hat{E}_{2}$ such that $(E_{i}+\hat{E}_{i})=f^{*}(A_{i})$, for $i=1,2$, and $(E_{1} \cup \hat{E}_{1})\cap (E_{2} \cup \hat{E}_{2})=\emptyset$. \\

\noindent \textit{\textbf{Step 2}}: Up to replacing $E_{1}$ with $\hat{E}_{1}$, $E_{2}$ or $\hat{E}_{2}$, there exists a smooth, prime divisor $G_{1}\subset X$ that is a $\mathbb{P}^{1}$-bundle with fiber $g_{1}\subset G_{1}$ such that $G_{1}\cdot g_{1}=-1$, $E_{1}\cdot g_{1}>0$, $G_{1}\neq E_{1}$, and $[g_{1}] \notin \mathcal{N}_{1}(E_{1},X)$. Moreover, $G_{1}$ dominates $Y$ and $f^{*}(A_{2})\cdot g_{1}>0$.

\begin{proof}[Proof of Step 2]
Suppose that there exists $D$ such that $\codim{\mathcal{N}_{1}(D,X)}=2$. We assume that $D=E_{1}$. Then $\codim{\mathcal{N}_{1}(E_{1},X)}=2$, so that being $\mathcal{N}_{1}(E_{1},X)\subseteq E_{2}^{\perp}\cap \hat{E}_{2}^{\perp}$, by Lemma \ref{div. ind} it follows that $\mathcal{N}_{1}(E_{1},X)= E_{2}^{\perp}\cap \hat{E}_{2}^{\perp}$.
 
Applying Proposition \ref{MMP} and Lemma \ref{iperpiano} to the divisor $E_{1}$, we get one smooth and prime divisor $G_{1}\subset X$ that is a $\mathbb{P}^{1}$-bundle with fiber $g_{1}$, such that $G_{1}\cdot g_{1}=-1$, $E_{1}\cdot g_{1}>0$, $G_{1}\neq E_{1}$, and $[g_{1}] \notin \mathcal{N}_{1}(E_{1},X)$ so that $H:=\mathcal{N}_{1}(E_{1},X)\oplus \mathbb{R}[g_{1}]$ is a hyperplane in $\mathcal{N}_{1}(X)$. 

Being $[g_{1}]\notin \mathcal{N}_{1}(E_{1},X)= E_{2}^{\perp}\cap \hat{E}_{2}^{\perp}$, then either $E_{2}\cdot g_{1}\neq0$ or $\hat{E}_{2}\cdot g_{1}\neq 0$. We observe that the intersection numbers $E_{2}\cdot g_{1}$ and $\hat{E}_{2}\cdot g_{1}$ cannot be negative. Indeed, if $E_{2}\cdot g_{1}< 0$, then $E_{2}=G_{1}$ and we get a contradiction because $E_{1}\cdot g_{1}> 0$, hence $G_{1}\cap E_{1}\neq \emptyset$, instead $E_{2}\cap E_{1}=\emptyset$ and similarly for $\hat{E}_{2}$. We can assume that $E_{2}\cdot g_{1}>0$, thus $f^{*}(A_{2})\cdot g_{1}>0$; in particular $E_{2}\cap G_{1}\neq \emptyset$. 

We show that $G_{1}\cdot e> 0$, where $e$ is the general fiber of $f\colon X\to Y$. By Proposition \ref{TEO1} (2), we recall that $e\equiv e_{2}+\hat{e}_{2}$. By contradiction, if $G_{1}\cdot e=0$, one has $G_{1}\cdot e_{2}= G_{1}\cdot \hat{e}_{2}=0$ (with the same method applied before, we deduce that the intersection numbers $G_{1}\cdot e_{2}$ and $G_{1}\cdot \hat{e}_{2}$ cannot be negative). Then there exists an irreducible curve $\bar{e}_{2}\subset G_{1}$ such that $\bar{e}_{2}\equiv e_{2}$, and applying Remark \ref{intersezione} to the divisors $G_{1}$ and $E_{1}$, we have $\bar{e}_{2}\equiv \lambda g_{1}+\mu C^{\prime}$ where $\lambda, \mu \in \mathbb{R}$, $\mu\geq 0$ and $C^{\prime}$ is an irreducible curve contained in $G_{1}\cap E_{1}$. The intersection with $E_{2}$ gives us $-1=\lambda E_{2}\cdot g_{1}$, so that $\lambda<0$. 

If we intersect with $\hat{E}_{2}$ we get $1=\lambda \hat{E}_{2}\cdot g_{1}$, so that $\hat{E}_{2}\cdot g_{1}<0$ that is a contradiction because as observed before $G_{1}\neq \hat{E}_{2}$. Thus $G_{1}\cdot e>0$, and hence $G_{1}$ dominates $Y$.\\

Suppose now that $\codim{\mathcal{N}_{1}(E_{i},X)}=\codim{\mathcal{N}_{1}(\hat{E}_{i},X)}=3$, for $i=1,2$. We choose one among these divisors $E_{i}, \hat{E}_{i}$, for instance $E_{1}$, and we apply to it Proposition \ref{MMP} and Lemma \ref{iperpiano}.
In this way, we obtain two disjoint prime divisors of $X$, $G_{1}$ and $G_{2}$, such that every  $G_{i}$ is a $\mathbb{P}^{1}$-bundle with fiber $g_{i}\subset G_{i}$, and $G_{i}\cdot g_{i}=-1$. Moreover, $E_{1}\cdot g_{i}>0$, $[g_{i}]\notin \mathcal{N}_{1}(E_{1},X)$, $G_{i}\neq E_{1}$ and $H:=\mathcal{N}_{1}(E_{1},X)\oplus \mathbb{R}[g_{1}]\oplus \mathbb{R}[g_{2}]$ is a hyperplane in $\mathcal{N}_{1}(X)$.

We observe that ${E}_{2}^{\perp}\neq {\hat{E}}_{2}^{\perp}$, because the divisors $E_{2}$ and $\hat{E}_{2}$ are not numerically proportional. Being $H$ a hyperplane, it will be different to at least one among ${E}_{2}^{\perp}$ and ${\hat{E}}_{2}^{\perp}$. We can assume that $H \neq {E}_{2}^{\perp}$. If $E_{2}\cdot g_{1}=E_{2}\cdot g_{2}=0$, then ${E}_{2}^{\perp}$ contains $\mathcal{N}_{1}(E_{1},X)$, $[g_{1}]$, $[g_{2}]$ and hence $H$, which is impossible. Up to exchanging $G_{1}$ and $G_{2}$, we can assume that $E_{2}\cdot g_{1}\neq 0$. It is not difficult to check that $E_{2}$ and $\hat{E}_{2}$ are different from $G_{1}$ and $G_{2}$, so that $E_{2}\cdot g_{i}\geq 0$, $\hat{E}_{2}\cdot g_{i}\geq 0$, for $i=1,2$. Finally, being $E_{2}\cdot g_{1}> 0$ and $\hat{E}_{2}\cdot g_{1}\geq 0$, one has that $(E_{2}+\hat{E}_{2})\cdot g_{1}=f^{*}(A_{2})\cdot g_{1}>0$. 
\end{proof}

\noindent \textit{\textbf{Step 3}}: There exists a smooth $\mathbb{P}^{1}$-fibration $\xi\colon Y\to Y^{\prime}$ where $Y^{\prime}$ is smooth and projective, and $\NE{(\xi)}=\mathbb{R}_{\geq0}[f(g_{1})]$.

\begin{proof}[Proof of Step 3]
Let $G_{1}\subset X$ be as in Step 2, and consider the restriction $f_{\mid_{G_{1}}}\colon G_{1}\to Y$. We observe that $f_{\mid_{G_{1}}}$ is a morphism such that $\dim{f(g_{1})}=1$. 
The statement easily follows applying Proposition \ref{smooth morphism} and Step 2. More precisely, by the proof of the same proposition (see \cite[Lemma 3.2.25]{CAS1}), we get $\NE{(\xi)}=\mathbb{R}_{\geq0}[g_{1}^{\prime}]$, where $g_{1}^{\prime}:=f(g_{1})$.  

\end{proof}
The situation is represented in Figure \ref{fig:fig3}.

\begin{figure}[ht]
	\centering
		\includegraphics[scale=1]{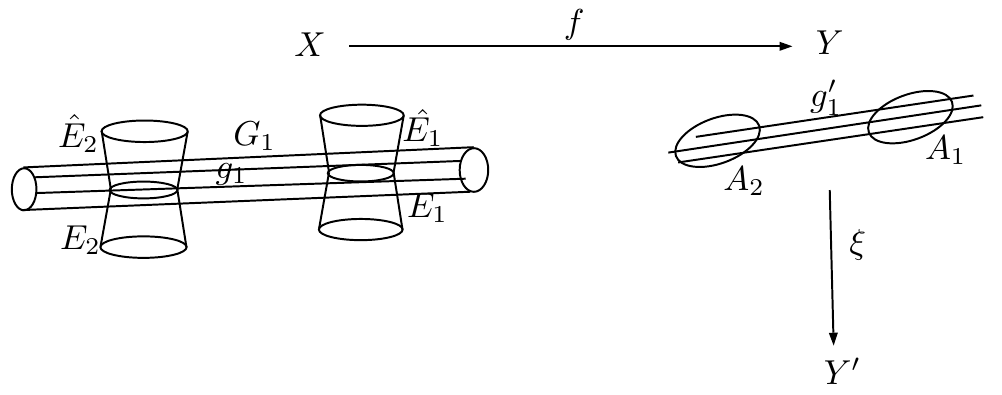}
	\caption{If $\rho_{X}-\rho_{Y}=3$, there exists a $\mathbb{P}^{1}$-bundle $G_{1}$ which dominates $Y$, and $g_{1}^{\prime}:=f(g_{1})$ spans the extremal ray whose contraction is $\xi$.}
	\label{fig:fig3}
\end{figure}

We are left to show that $g$ is smooth to get $(2)$, in fact by Corollary \ref{onlyred} this will imply that both $Y$ and $Y^{\prime}$ are Fano. From now on, we assume by contradiction that $g$ is singular.\\

\noindent \textit{\textbf{Step 4}}: The numerical classes $[E_{1}], [\hat{E_{1}}], [E_{2}], [\hat{E_2}]$ are linearly dependent in $\mathcal{N}^{1}(X)$ and there is a fibration $\Phi\colon Y\to \mathbb{P}^{1}$ which sends the divisors $A_{1},A_{2}, \bigtriangleup_{g}$ to points. 
 
\begin{proof}[Proof of Step 4]
Let us denote by $\tilde{\bigtriangleup}_{g}$ the inverse image in $X$ of $\bigtriangleup_{g}$. It is disjoint from $E_{i}$, $\hat{E}_{i}$, for $i=1,2$, thus $\mathcal{N}_1(\tilde{\bigtriangleup}_{g},X)\subseteq E_{1}^{\perp}\cap {\hat{E}_{1}}^{\perp}\cap E_{2}^{\perp}\cap {\hat{E}_{2}}^{\perp}$. Since $\delta_{X}\leq 3$, we deduce that the numerical classes $[E_{1}], [\hat{E_{1}}], [E_{2}], [\hat{E_2}]$ are linearly dependent in $\mathcal{N}^{1}(X)$.

Now we proceed as in the proof of (1): we get $E_{1}+\hat{E_{1}}\equiv a(E_{2}+\hat{E_{2}})$, $a \in \mathbb{R}$ and again $A_{1}\equiv a A_{2}$. The nef divisors $A_{1}, A_{2}$ give a contraction $\Phi\colon Y\to \mathbb{P}^{1}$ such that $\Phi(A_{i})=\{pt\}$. Since $\bigtriangleup_{g}\cap A_{1}=\emptyset$, the image of every component of $\bigtriangleup_{g}$ is also a point.
\end{proof}


\noindent \textit{\textbf{Step 5}}: $Y\cong \mathbb{P}^{1}\times Y^{\prime}$ and $\bigtriangleup_{g}=\{points\}\times Y^{\prime}$.

\begin{proof}[Proof of Step 5]
We have two maps from $Y$, \textit{i.e.} $\Phi\colon Y\to \mathbb{P}^{1}$ and $\xi\colon Y\to Y^{\prime}$, where $\Phi$ is finite on the fibers of $\xi$, since $A_{2}\cdot f(g_{1})>0$. 
 
We first prove that a general fiber $A_{0}$ of $\Phi$ is a Fano variety. Consider the fiber type $K$-negative contraction $\Psi:=\Phi\circ f\colon X\stackrel{f}{\to}Y\stackrel{\Phi}{\to}\mathbb{P}^{1}$, whose general fiber is $f^{-1}(A_{0})$ and is smooth and Fano. One has that $f_{\mid_{f^{-1}(A_{0})}} \colon f^{-1}(A_{0})\to A_{0}$ is a $\mathbb{P}^{1}$-bundle: indeed  $\bigtriangleup_{f}=A_{1}\sqcup A_{2}\sqcup \bigtriangleup_{g}$ is a union of fibers of $\Phi$, so $A_{0}\cap \bigtriangleup_{f}=\emptyset$. Then, by Corollary \ref{Fano} it follows that $A_{0}$ is Fano.

Now, using that $\NE{(\Phi)}$ is generated by finitely many classes of rational curves (see \cite[Lemma 2.6]{CAS3}) and that the general fiber of $\Phi$ is a Fano manifold, the same proof of \cite[Lemma 4.9]{CAS4} yields that $Y\cong \mathbb{P}^{1}\times Y^{\prime}$ and $\bigtriangleup_{g}=\{points\}\times Y^{\prime}$.
\end{proof}

\noindent\textit{\textbf{Step 6}}: We reach a contradiction.

\begin{proof}[Proof of Step 6]
We have already shown that $\bigtriangleup_{g}=\{points\}\times Y^{\prime}$, thus $\bigtriangleup_{g}$ is smooth, hence by Corollary \ref{Fano} one has that  $\Sing{(\bigtriangleup_{f})}=\emptyset$ and $Y$ is Fano. Being $Y\cong \mathbb{P}^{1}\times Y^{\prime}$, $Y^{\prime}$ is Fano too, so that each connected component of $\bigtriangleup_{g}$ is simply connected.
Since $g$ is elementary, by the proof of \cite[Proposition 4.7 $(1)$]{MM} we know that $\bigtriangleup_{g}$ does not have smooth and simply connected irreducible components.  
Then we reach a contradiction and this concludes the proof of $(2)$.
%

\end{proof}

\noindent \textbf{\textit{Proof of (3)}}. By assumption $\rho_{X}-\rho_{Y}=2$, so there exists one pair of prime divisors $E_{1}, \hat{E}_{1}$ as in Proposition \ref{smoothcase} such that $E_{1}+\hat{E}_{1}=f^{*}(A_{1})$. Moreover, by our assumption, $\bigtriangleup_{g}\neq \emptyset$.  Let ${\tilde{{\bigtriangleup}^{\prime}}}_{g}$ be the inverse image in $X$ of $\bigtriangleup_{g}$. We work with an irreducible component of this divisor that we call $\tilde{\bigtriangleup}_{g}$. Notice that $\tilde{\bigtriangleup}_{g}\cap (E_{1}\cup \hat{E}_{1})=\emptyset$.

The proof adopts the same techniques used in the previous point. If $X\cong S\times T$ where $\dim{S}=2$, it is easy to check the statement by induction as in Step 1. Assume that $X\ncong S\times T$, so that by Theorem \ref{prodotto} we know that $\delta_{X}\leq 3$. On the other hand, using Lemma \ref{delta}, we find that $\delta_{X}\geq 2$. 

If $\codim{\mathcal{N}_{1}(\tilde{\bigtriangleup}_{g},X)}=2$ one has $\mathcal{N}_{1}(\tilde{\bigtriangleup}_{g},X)=E_{1}^{\perp}\cap {\hat{E}_{1}}^{\perp}$ and applying Proposition \ref{MMP} and Lemma \ref{iperpiano} to $\tilde{\bigtriangleup}_{g}$, we find a smooth prime divisor $G_{1}$ of $X$ which is a $\mathbb{P}^{1}$-bundle with fiber $g_{1}$ such that $G_{1}\cdot g_{1}=-1$, $\tilde{\bigtriangleup}_{g}\cdot g_{1}>0$, $G_{1}\neq \tilde{\bigtriangleup}_{g}$, and a hyperplane of $\mathcal{N}_{1}(X)$ that is $\mathcal{N}_{1}(\tilde{\bigtriangleup}_{g},X)\oplus \mathbb{R}[g_{1}]$. Using the same method as in Step 2 (replacing $E_{1}$ with $\tilde{\bigtriangleup}_{g}$), we deduce that $G_{1}$ dominates $Y$, and $f^{*}(A_{1})\cdot g_{1}>0$. 

Otherwise, if $\codim{\mathcal{N}_{1}(\tilde{\bigtriangleup}_{g},X)}=3$, again by Proposition \ref{MMP} and Lemma \ref{iperpiano} applied to the divisor $\tilde{\bigtriangleup}_{g}$, we get two disjoint prime divisors of $X$, $G_{1}$ and $G_{2}$. Every $G_{i}$ is a $\mathbb{P}^{1}$-bundle with fiber $g_{i}\subset G_{i}$, and $G_{i}\cdot g_{i}=-1$. Moreover, $\tilde{\bigtriangleup}_{g}\cdot g_{i}>0$, $[g_{i}]\notin \mathcal{N}_{1}(\tilde{\bigtriangleup}_{g},X)$, $G_{i}\neq \tilde{\bigtriangleup}_{g}$, and we find the hyperplane $\mathcal{N}_{1}(\tilde{\bigtriangleup}_{g},X)\oplus \mathbb{R}[g_{1}]\oplus \mathbb{R}[g_{2}]$. Proceeding again as in Step 2, we prove that the $\mathbb{P}^{1}$-bundle $G_{i}$ ($i=1$ or $i=2$) dominates $Y$. We can assume that it is $G_{1}$.

In any case, we proceed as done in Step 3 to find the smooth $\mathbb{P}^{1}$-fibration $\xi$ (we replace $A_{2}$ with $A_{1}$ and $A_{1}$ with $\bigtriangleup_{g}$). The situation is similar to that represented in Figure \ref{fig:fig3}. This concludes the proof of $(3)$ and hence the proof of the theorem. 
\end{proof}

Now, looking at the fibers of the conic bundle, we deduce the relation between $\rho_{X}-\rho_{Y}$ and the singular fibers of $f$.

\begin{corollary} \label{fibre} Let $f\colon X\to Y$ be a Fano conic bundle. If $\rho_{X}-\rho_{Y}\geq 3$, then $f$ has only reduced fibers. If $\rho_{X}-\rho_{Y}=2$ and f has non-reduced fibers, then $Y$ has a smooth $\mathbb{P}^1$-fibration. 
\end{corollary}
\begin{proof} By Proposition \ref{smoothcase} $(2)$, we recall that the fibers of $f$ over the divisors $A_{i}$ are singular but reduced, hence $f^{-1}(y)$ is non-reduced if and only if $g^{-1}(y)$ is.  

If $\rho_{X}-\rho_{Y}\geq 3$, by Theorem \ref{TEOPRINC} $(2)$, $g$ is a smooth contraction so there are no non-reduced fibers.

If $\rho_{X}-\rho_{Y}=2$ and $f$ has non-reduced fibers, then $\bigtriangleup_{g}\neq \emptyset$, hence the statement follows by Theorem \ref{TEOPRINC} $(3)$. 
\end{proof}

We conclude this subsection proving Theorem \ref{ThmINTRODUZ}. 
\begin{proof} [Proof of Theorem \ref{ThmINTRODUZ}] Statement $(1)$ is the same of Theorem \ref{TEOPRINC}. 

To prove $(2)$ it is enough to observe that if $\rho_{X}-\rho_{Y}=3$, then applying Theorem \ref{TEOPRINC} and Corollary \ref{onlyred} it follows that $f$ has only reduced fibers. The rest of statement was already shown in Theorem \ref{TEOPRINC}.

We show $(3)$; assume that $\rho_{X}-\rho_{Y}=2$ and $Y$ is not Fano. Take a factorization of $f$ as in Proposition \ref{smoothcase} and denote by $g$ the elementary conic bundle in this factorization. By Corollary \ref{Fano} one has that $f$ has non-reduced fibers. As observed in the proof of Corollary \ref{fibre}, this means that $g$ has non-reduced fibers, so that $g$ is singular and by Theorem \ref{TEOPRINC} $(3)$ it follows that there exists a smooth $\mathbb{P}^{1}$-fibration $\xi\colon Y\to Y^{\prime}$, where $Y^{\prime}$ is smooth. It remains to prove that $\rho_{Y}\geq 3$.

Being $g$ singular, by Proposition \ref{smoothcase} we get two disjoint divisors in $Y$, $A_{1}$ and $\bigtriangleup_{g}$. 
We prove that $\dim{\mathcal{N}_{1}(\bigtriangleup_{g},Y)}\geq 2$. Since $Y$ is not Fano, there exists an extremal ray $R$ of $Y$ such that $-K_{Y}\cdot R\leq 0$, and by Proposition \ref{PropWis} it follows that $\Lo{(R)}\subseteq \Sing{(\bigtriangleup_{f})}=\Sing{(\bigtriangleup_{g})}$. Assume by contradiction that $\dim{\mathcal{N}_{1}(\bigtriangleup_{g},Y)}=1$. Then all curves in $\bigtriangleup_{g}$ are numerically proportional in $Y$, thus they are all contained in $\Lo{(R)}$ and we reach a contradiction because $\Sing{(\bigtriangleup_{g})}\subsetneq \bigtriangleup_{g}$. Thus $\dim{\mathcal{N}_{1}(\bigtriangleup_{g},Y)}\geq 2$ and since $\mathcal{N}_{1}(\bigtriangleup_{g},Y)\subseteq A_{1}^{\perp}$, we get $\rho_{Y}\geq 3$.
\end{proof}

\section{Related results and Examples}
\noindent As usual, given a Fano conic bundle $f\colon X\to Y$, set $r:=\rho_{X}-\rho_{Y}$. In this section we give some relevant examples and discuss some applications of our results.  To conclude, we will focus on the case $r=2$.
  
\subsection{Example of Fano conic bundle with $\mathbf{r=3}$ where $X$ is not a product} \label{noproduct}
Let $X$ be the Fano threefold N.$18$ in \cite[$\S$12.8]{PAR}, which is not a product. This $X$ is obtained from the blow-up of three smooth rational curves in $Z=\mathbb{P}_{\mathbb{P}^{1}}(\mathcal{O}^{\oplus 2}\oplus \mathcal{O}(1))$, and $\rho_{X}=5$. Mori and Mukai proved the existence of a conic bundle $X\to \mathbb{F}_{1}$ (see table 5 of \cite{MM3}). Notice that the target $\mathbb{F}_{1}$ has a smooth $\mathbb{P}^{1}$-fibration $\xi\colon \mathbb{F}_{1}\to \mathbb{P}^{1}$, as stated by Theorem \ref{TEOPRINC}.
\subsection{Examples of elementary Fano conic bundles} \label{elementaryex}
Wi$\acute{\text{s}}$niewski gives an example of a Fano conic bundle $f\colon X\to Y$ in which $Y$ is not a Fano variety (see \cite[Example $\S 4$]{WIS}). Here our goal is to show, using Corollary \ref{esempio}, that $f$ is an elementary conic bundle.

Wi$\acute{\text{s}}$niewski's construction was generalized by Debarre (see \cite[Example 3.16, (3)]{DEB}) and goes as follows. Let $a$ and $b \in \mathbb{Z}^{+}$, set $m=2a+2b-1$, let $\mathcal{E}$ be the vector bundle associated to the locally free sheaf $\mathcal{O}_{\mathbb{P}^{m}}\oplus \mathcal{O}_{\mathbb{P}^{m}}(2a)\oplus \mathcal{O}_{\mathbb{P}^{m}}(2b)$, and let $Y=\mathbb{P}_{\mathbb{P}^{m}}(\mathcal{E})$ (notice that in his example, Wi$\acute{\text{s}}$niewski takes $a=b=1$, so that $Y$ has dimension 5).

One has $\rho_{Y}=2$, and $-K_{Y}=3D$, where $D$ is associated with the line bundle $\mathcal{O}_{Y}(1)$. In particular, $-K_{Y}$ is not ample, because $D$ is a nef and big divisor, not ample (see for instance \cite[$\S$2.3, Lemma 2.3.2]{LAZ}).

Wi$\acute{\text{s}}$niewski gives a rank $3$ vector bundle $\mathcal{M}$ on $Y$ and a smooth divisor $X\subset \mathbb{P}_{Y}(\mathcal{M})$ such that $X\to Y$ is a Fano conic bundle. The divisor $X$ has a non-empty base locus in $\mathbb{P}_{Y}(\mathcal{E})$.

Now, by Corollary \ref{esempio} it follows that $r=1$ because $\rho_{Y}=2$ and also because $Y$ does not have a smooth $\mathbb{P}^{1}$-fibration.

\subsection{Complements on case $\mathbf{r}=2$.} \label{esempiosing}
If $f\colon X\to Y$ is a Fano conic bundle where $r:=\rho_{X}-\rho_{Y}=2$, then by Proposition \ref{smoothcase} we get two factorizations for $f$. Take one among these, and let $g$ be the elementary conic bundle of this factorization.

The author knows no example of such a Fano conic bundle $f\colon X\to Y$ with $g$ singular.
Assume that there exists such an example, and that $\delta_{X}\geq 4$. By Theorem \ref{prodotto}, we know that $X\cong S\times T$ where $\dim{S}=2$ and by induction on the dimension (as done in Step 1 in the proof of Theorem \ref{TEOPRINC}), we find that $X$ is a product between a finite number of del Pezzo surfaces and another Fano variety $Z$ with $\delta_{Z}\leq 3$, $f$ is induced by a conic bundle $f^{\prime}\colon Z\to Z^{\prime}$ with $\rho_{Z}-\rho_{Z^{\prime}}=2$, and the elementary conic bundle in the factorization of $f$ is singular.  
Therefore, to study Fano conic bundles $f\colon X\to Y$, where $r=2$ and $g$ is singular, it makes sense to focus on the case in which $\delta_{X}\leq 3$, and hence $X\ncong S\times T$ by Proposition \ref{prodotto}.

\begin{proposition} We use the notation as in Proposition \ref{smoothcase}. Assume that $X\not\cong S\times T$, where $S$ is a del Pezzo surface. Let $f\colon X\to Y$ be a Fano conic bundle with $r=2$ and such that $\bigtriangleup_{g}\neq \emptyset$. Then $\delta_{X}\in \{2,3\}$ and the following hold:
\begin{enumerate}[(1)]
\item $\bigtriangleup_{g}$ is not a section of the smooth $\mathbb{P}^{1}$-fibration $\xi\colon Y\to Y^{\prime}$ of Theorem \ref{TEOPRINC} (3);
\item the numerical classes $[\bigtriangleup_{g}]$ and $[A_{1}]$ are linearly independent in $\mathcal{N}^{1}(Y)$;
\item if $\triangle$ is an irreducible component of $\bigtriangleup_{g}$, then for every divisor $D\subset Y$ one has $D\cap (\triangle\cup A_{1})\neq \emptyset$;
\item $\bigtriangleup_{g}$ is connected, so that $\bigtriangleup_{f}$ has exactly two connected components: $A_{1}$ that is smooth and $\bigtriangleup_{g}$ that contains $\Sing{(\bigtriangleup_{f})}$.
\end{enumerate}
\end{proposition}
\begin{proof}
Using Theorem \ref{prodotto} and Lemma \ref{delta} we find that either $\delta_{X}=3$ or $\delta_{X}= 2$. We prove $(1)$ by contradiction: if $\bigtriangleup_{g}$ is a section of $\xi$, one has that $\bigtriangleup_{g}\cong Y^{\prime}$, hence $\bigtriangleup_{g}$ is smooth and simply connected which is impossible (see Step 6 in the proof of Theorem \ref{TEOPRINC}). 

We show $(2)$. Assume by contradiction that $\bigtriangleup_{g}\equiv a A_{1}$, with $a\in \mathbb{R}$. Being $\bigtriangleup_{g}\cap A_{1}=\emptyset$, $\bigtriangleup_{g}$ and $A_{1}$ are nef divisors of $Y$ and as we explained in the proof of Theorem \ref{TEOPRINC} $(1)$, they give a contraction $\Phi\colon Y\to \mathbb{P}^{1}$ that sends them to points. We can proceed as done in Step 5 and in Step 6 in the proof of Theorem \ref{TEOPRINC}. We get a contradiction in the same way, because every irreducible component of $\bigtriangleup_{g}$ would be smooth and simply connected.


To prove $(3)$, assume that there exists a divisor $D\subset Y$ and an irreducible component $\triangle$ of $\bigtriangleup_{g}$ such that $D\cap (\triangle\cup A_{1})=\emptyset$. Then $\mathcal{N}_{1}(\triangle,Y)\subseteq D^{\perp}$ and $\mathcal{N}_{1}(A_{1},Y)\subseteq D^{\perp}$. 

By the proof of Theorem \ref{TEOPRINC}, we get $A_{1}\cdot g_{1}^{\prime}=(E_{1}+\hat{E}_{1})\cdot g_{1}>0$, where $g_{1}^{\prime}=f(g_{1})$ spans the extremal ray whose contraction is the smooth $\mathbb{P}^{1}$-fibration $\xi$ stated by the same theorem. Hence \cite[Lemma 3.2]{OCCHETTA} implies that $\mathcal{N}_{1}(Y)=\mathcal{N}_{1}(A_{1},Y)\oplus \mathbb{R}[g_{1}^{\prime}]$. The same holds taking the divisor $\triangle$, thus $\mathcal{N}_{1}(Y)=\mathcal{N}_{1}(\triangle,Y)\oplus \mathbb{R}[g_{1}^{\prime}]$. Both $\mathcal{N}_{1}(\triangle,Y)$ and $\mathcal{N}_{1}(A_{1},Y)$ have codimension 1 in $\mathcal{N}_{1}(Y)$, so that $\mathcal{N}_{1}(\triangle,Y)=\mathcal{N}_{1}(A_{1},Y)=D^{\perp}$. But $\triangle\cap A_{1}=\emptyset$, thus $\triangle$ and $A_{1}$ are nef divisors, and using that $\mathcal{N}_{1}(\triangle,Y)=\mathcal{N}_{1}(A_{1},Y)$ it is easy to check that for every irreducible curve $C\subset \triangle$, one has $\triangle\cdot C=0$ and the same holds for $A_{1}$. Therefore, $\triangle$ and $A_{1}$ give the same contraction $\Phi\colon Y\to \mathbb{P}^{1}$ that sends them to points and we reach a contradiction as in $(2)$. 

Finally $(4)$ is a straightforward consequence of $(3)$ and of Proposition \ref{smoothcase} $(2)$.
\end{proof}

%
%



\nocite{*}
\small
\footnotesize
\bibliographystyle{plain} 
\bibliography{biblio}

\begin{thebibliography}{10}

\bibitem{Alexeev}
V.~ALEXEEV.
\newblock Higher dimensional analogues of stable curves.
\newblock {\em Proceedings of {M}adrid {ICM}2006, Eur. Math. Soc. Pub. House},
  pages 515--536, 2006.

\bibitem{ANDO}
T.~ANDO.
\newblock On extremal rays of the higher-dimensional varieties.
\newblock {\em Invent. Math.}, 81:347--357, 1985.

\bibitem{B}
A.~BEAUVILLE.
\newblock Varieties de {P}rym et {J}acobiennes interm$\acute{\text{e}}$diaires.
\newblock {\em Ann. Scient. $\acute{E}$c. Norm. Sup.}, pages 309--391, 1977.

\bibitem{BIRKAR}
C.~BIRKAR, P.~CASCINI, C.D. HACON, and J.~MCKERNAN.
\newblock Existence of minimal models for varieties of log general type.
\newblock {\em J. Amer. Math. Soc}, 23:405--468, 2010.

\bibitem{CAS3}
C.~CASAGRANDE.
\newblock Quasi-elementary contractions of {F}ano manifolds.
\newblock {\em Compos. Math.}, 144:1429--1460, 2008.

\bibitem{CAS4}
C.~CASAGRANDE.
\newblock On {F}ano manifolds with a birational contraction sending a divisor
  to a curve.
\newblock {\em Michigan Math. J.}, 58:783--805, 2009.

\bibitem{CAS1}
C.~CASAGRANDE.
\newblock On the {P}icard number of divisors in {F}ano manifolds.
\newblock {\em Ann. Scient. $\acute{E}$c. Norm. Sup.}, 45:363--403, 2012.

\bibitem{CAS2}
C.~CASAGRANDE.
\newblock On some {F}ano manifolds admitting a rational fibration.
\newblock {\em J. London Math. Soc.}, 90:1--28, 2014.

\bibitem{DEB}
O.~DEBARRE.
\newblock {\em Higher-dimensional algebraic geometry}.
\newblock Universitext, Springer-Verlag, New York, 2001.

\bibitem{EJIRI}
S.~EJIRI.
\newblock Positivity of anti-canonical divisors and {F}-purity of fibers.
\newblock ArXiv:1604.02022v2, 2016.

\bibitem{FUJ}
O.~FUJINO and Y.~GONGYO.
\newblock On images of weak {F}ano manifolds.
\newblock {\em Mathematische Zeitschrift}, 270(1-2):531--544, 2012.

\bibitem{HART}
R.~HARTSHORNE.
\newblock {\em Algebraic Geometry}.
\newblock Springer-Verlag New York, 1997.

\bibitem{HU}
Y.~HU and S.~KEEL.
\newblock Mori dream space and {GIT}.
\newblock {\em Michigan Math. J.}, 48:331--348, 2000.

\bibitem{K}
J.~KOLL{\'A}R, Y.~MIYAOKA, and S.~MORI.
\newblock Rational connectedness and boundedness of {F}ano manifolds.
\newblock {\em Journal of Differential Geometry}, 36(3):765--779, 1992.

\bibitem{KOLLAR}
J.~KOLL\'AR and S.~MORI.
\newblock {\em Birational geometry of algebraic varieties}.
\newblock Cambridge Tracts in Math.,134, Cambridge Univ. Press, Cambridge,
  1998.

\bibitem{LAZ}
R.~LAZARSFELD.
\newblock {\em Positivity in algebraic geometry $I$}.
\newblock Ergebn. Math. Grenzg. Springer, 2004.

\bibitem{MM3}
S.~MORI and S.~MUKAI.
\newblock Classification of {F}ano 3-folds with $b_{2}\geqslant 2$.
\newblock {\em Manuscripta Math.}, 36:147--162, 1981.
\newblock \textit{Erratum}, Manuscripta Math., 110:407,2003.

\bibitem{MM2}
S.~MORI and S.~MUKAI.
\newblock On {F}ano 3-folds with $b_{2}\geqslant 2$.
\newblock {\em Advanced Studies in Pure Mathematics 1}, pages 101--129, 1983.

\bibitem{MM}
S.~MORI and S.~MUKAI.
\newblock Classification of {F}ano 3-folds with $b_{2}\geqslant 2$, {I}.
\newblock In {\em Algebraic and Topological Theories}, pages 496--545, 1985.

\bibitem{NOCE}
G.~DELLA NOCE.
\newblock On the {P}icard number of singular {F}ano varieties.
\newblock {\em Int. Math. Res. Not.}, pages 955--990, 2014.

\bibitem{OCCHETTA}
G.~OCCHETTA.
\newblock A characterization of products of projective spaces.
\newblock {\em Canad. Math. Bull}, 49:270--280, 2006.

\bibitem{PAR}
A.N. PARSHIN and I.R. SHAFAREVICH.
\newblock {\em Algebraic {G}eometry {V}: {F}ano {V}arieties}.
\newblock Encyclopaedia of Mathematical Sciences, Springer, 1999.

\bibitem{SARK}
V.G. SARKISOV.
\newblock On conic bundle structures.
\newblock {\em Mathematics of the USSR-Izvestiya}, 20:355--390, 1983.

\bibitem{WIS}
J.A. WI\'SNIEWSKI.
\newblock On contractions of extremal rays of {F}ano manifolds.
\newblock {\em J. reine angew. Math.}, 417:141--157, 1991.

\end{thebibliography}

\end{document}